\documentclass[12pt]{amsart}
\usepackage[margin=1.2in]{geometry}

\title{CUR Decompositions, Approximations, and Perturbations}

\usepackage{hyperref}

\usepackage{hyperref}
\usepackage{tablefootnote}
\usepackage{graphicx}
\usepackage{subcaption}
\usepackage{amsthm}
\usepackage{dsfont}
\usepackage{amssymb}
\usepackage{enumerate}
\usepackage{graphicx}
\usepackage{float}
\usepackage{bbm}
\usepackage{amsmath}
\usepackage{comment}
\usepackage{hyperref}
\usepackage{listings}
\usepackage{color}
\usepackage{ulem}
\usepackage[dvipsnames]{xcolor}
\usepackage[lined,boxed,commentsnumbered]{algorithm2e}
\allowdisplaybreaks

\hypersetup{
	colorlinks,
	citecolor=blue,
	filecolor=blue,
	linkcolor=blue,
	urlcolor=blue,
	hyperfootnotes=false
}

\newtheorem{theorem}{Theorem}[section]
\newtheorem{proposition}[theorem]{Proposition}
\newtheorem{lemma}[theorem]{Lemma}

\newtheorem{corollary}[theorem]{Corollary}

\newtheorem{thmnum}{Theorem}

\newtheorem{cornum}{Corollary}

\theoremstyle{definition}
\newtheorem{example}[theorem]{Example}

\newtheorem{problem}{Problem}
\newtheorem{experiment}{Experiment}
\newtheorem{remark}[theorem]{Remark}

\newcommand{\R}{\mathbb{R}}

\newcommand{\K}{\mathbb{K}}

\newcommand{\E}{\mathbb{E}}
\newcommand{\Prob}{\mathbb{P}}

\newcommand{\sspan}{\textnormal{span}}

\newcommand{\rank}{{\rm rank\,}}

\newcommand{\argmin}{\text{argmin}}

\newcommand{\eps}{\varepsilon}

\begin{document}

\author{Keaton Hamm}
\address{Department of Mathematics, University of Arizona, Tucson, AZ 85719 USA}
\email{hamm@math.arizona.edu}

\author{Longxiu Huang}
\address{Department of Mathematics, Vanderbilt University, Nashville, TN 37240 USA}
\email{longxiu.huang@vanderbilt.edu}



\begin{abstract}

This article discusses a useful tool in dimensionality reduction and low-rank matrix approximation called the CUR decomposition.  Various viewpoints of this method in the literature are synergized and are compared and contrasted; included in this is a new characterization of exact CUR decompositions.  A novel perturbation analysis is performed on CUR approximations of noisy versions of low-rank matrices, which compares them with the putative CUR decomposition of the underlying low-rank part.  Additionally, we give new column and row sampling results which allow one to conclude that a CUR decomposition of a low-rank matrix is attained with high probability.  We then illustrate the stability of these sampling methods under the perturbations studied before, and provide numerical illustrations of the methods and bounds discussed.
\end{abstract}

\maketitle


\section{Introduction}

In many data analysis applications, two key tools are dimensionality reduction and compression, in which data obtained as vectors in a high dimensional Euclidean space are approximated in a basis or frame which spans a much lower dimensional space than the ambient space of the data (reduction) or a sketch of the total data matrix is made and stored in memory (compression).  Without such steps as a preconditioner to further analysis, many problems would be intractable.  However, one must balance the approximation method with the demand that any results obtained from the approximate versions of the data be readily interpreted by domain experts. This task can be challenging, and many well-known methods (for instance PCA) allow for great approximation and compression of the data, but at the cost of inhibiting interpretation of the results using the underlying physics or application.

One way around this difficulty is to attempt to utilize the \textit{self-expressiveness} of the data, which is the notion that oftentimes data is better represented in terms of linear combinations of other data points rather than in some abstract basis.  In many applications data is self-expressive, and methods based on this assumption achieve rather good results in various machine learning tasks (as a particular example, we refer the reader to the Sparse Subspace Clustering algorithm of Elhamifar and Vidal \cite{SSC}).  The question then is: how may one use self-expressiveness to achieve dimensionality reduction? Mahoney and Drineas \cite{DMPNAS} argue for the representation of a given data matrix in terms of \textit{actual columns and rows} of the matrix itself.  The idea is that rather than do something like PCA to transform the data into an eigenspace representation, one attempts to choose the most representative columns and rows which capture the essential information of the initial matrix.  Thus enters the \textit{CUR Decomposition}.

There are two distinct starting points in most of the literature revolving around the CUR decomposition (also called (pseudo)skeleton approximations \cite{DemanetWu,Goreinov}).  The first is as an exact matrix \textit{decomposition}, or \textit{factorization}: given an arbitrary, and possibly complicated matrix $A$, one may desire to decompose $A$ into the product of 2 or more factors, each of which is ``easier" to understand, store, or compute with than $A$ itself. The exact decomposition (see Theorem \ref{THM:CUR} for a formal statement) says that $A=CU^\dagger R$ if $\rank(U)=\rank(A)$, where $C$ and $R$ are column and row submatrices of $A$, respectively, and $U$ is their intersection.   The second starting point is to find a \textit{low-rank approximation} to a given matrix. This is along the lines of the work of Drineas, Kannan, Mahoney, and others \cite{DKMIII,DM05,DMM08}, and stems from the considerations of interpretability above.  Additionally, this is the typical vantage point of much of randomized linear algebra \cite{tropp}.  In this setting, the CUR approximation of a matrix is typically given by $A\approx CC^\dagger AR^\dagger R$, where $CC^\dagger$ and $R^\dagger R$ are orthogonal projections onto the subspaces spanned by given columns and rows, respectively.

These perspectives are not typically addressed in the literature together (indeed, the first perspective is taken but rarely).  The first part of this article (Section \ref{SEC:Viewpoints}) is an exposition of the two CUR viewpoints presented above, and highlights their similarities and differences; in doing so we give some new characterizations for exact CUR decompositions.  We also address two issues of import to CUR:  first, we undertake a novel perturbation analysis for CUR approximations of a flavor different than that in the existing literature.  Second, we give guarantees on random column and row sampling procedures which ensure that an exact CUR decomposition of a low-rank matrix is obtained with high probability, and then prove stability of this result under perturbations of the sampling probabilities used.  We then combine these two considerations to give guarantees on when sampling columns and rows of a noisy low-rank matrix gives an exact CUR decomposition for the low-rank part.

\section{Main Results}

The main theoretical results of this article are as follows (informally stated for now, but we point out the full statements later on): first in comparing CUR decompositions to CUR approximations, we find that in the exact case they are the same, but moreover, we derive several new equivalent conditions to obtaining an exact CUR decomposition.

\subsection{Equivalences for Exact CUR Decompositions}
\begin{thmnum}[Theorem  \ref{THM:Characterization} and Proposition \ref{PROP:Udagger}]
Let $C$ and $R$ be column and row submatrices of $A\in\K^{m\times n}$ (possibly with repeated columns/rows), and let $U$ be their intersection.  Then the following are equivalent:
\begin{enumerate}[(i)]
    \item $\rank(U)=\rank(A)$
    \item $A=CU^\dagger R$
    \item $A = CC^\dagger AR^\dagger R$
    \item $A^\dagger = R^\dagger UC^\dagger$
    \item $\rank(C)=\rank(R)=\rank(A)$.
\end{enumerate}
Moreover, if any of the equivalent conditions above hold, then $U^\dagger = C^\dagger AR^\dagger$.
\end{thmnum}
This theorem reconciles the exact CUR decomposition $A=CU^\dagger R$ with what is often called the CUR decomposition, but which we call here a CUR approximation, given by $A\approx CC^\dagger AR^\dagger R$.  The latter is the best CUR approximation in a sense that is made precise in the sequel.  The moreover part (and the equivalence of \ref{ITEM:Adagger}) is interesting in its own right because it is not generally the case that $(AB)^\dagger = B^\dagger A^\dagger.$

\subsection{Perturbation analysis for CUR approximations}

Secondly, in many applications, matrices are well-approximated to be low rank (see \cite{LowRank}).  Consequently, we often observe $\tilde{A}=A+E$, where $A$ is low rank, but $E$ is some (hopefully) small perturbation.  Given a CUR \textit{approximation} of $\tilde{A}$, we consider what happens to the underlying CUR \textit{decomposition} of $A$.  Namely, if $\tilde{C}$ and $\tilde{R}$ are column and row submatrices of $\tilde{A}$, and $\tilde{U}$ is the matrix of their intersection, such that $C$, $R$, and $U$ are the corresponding submatrices of $A$, then we find the following (see Section \ref{SEC:Notation} for precise definitions).
\begin{thmnum}[Theorem \ref{THM:PB}]\label{THM:IntroA}
Let $\tilde{A} = A+E$ with $\rank(A)=k$.  For any unitarily invariant, submultiplicative, normalized, uniformly generated matrix norm $\|\cdot\|$, for sufficiently small $\|E\|$,
\[\|A-\tilde{C}\tilde{U}_k^\dagger\tilde{R}\| \leq \|A-CU^\dagger R\| + O(\|E\|)+O(\|A^\dagger\|\|E\|^2). \]
\end{thmnum}
This is the first perturbation analysis of this form which compares CUR approximations of noisy versions of low-rank matrices to the underlying CUR decomposition of their low-rank part.  Indeed, if columns and rows are chosen such that $A=CU^\dagger R$ (see Theorem \ref{THM:IntroB}) then Theorem \ref{THM:IntroA} says that the CUR approximation of $A$ based on noisy columns and rows is bounded in norm by $O(\|E\| +\|A^\dagger\|\|E\|^2)$.  Specific and deterministic bounds for the big-O constant are given in the sequel.  One motivation for this analysis is the main theorem of \cite{AHKS}, which shows that the CUR decomposition of data matrices whose columns come from unions of subspaces can be used to solve the \textit{subspace clustering problem}.  The algorithm therein proposes a solution in the case that the data is noisy; however, there is no theoretical guarantee which guarantees success in the presence of noise.  Consequently, the perturbation bounds here may prove useful in that context.

\subsection{Column and Row Sampling and Stability}

Finally, we discuss the problem of how to select columns and rows to obtain an exact CUR decomposition of a low rank matrix, and prove that this procedure is stable under small perturbations. The initial result is obtained from some established results of Rudelson and Vershynin \cite{Rudelson_2007}.
\begin{thmnum}[Theorem \ref{THM:ColRowChnM}]\label{THM:IntroB}
If $A$ has rank $k$, then sampling $O(k\log k)$ columns and rows of $A$ independently with replacement according to column and row lengths, respectively, implies that $A = CU^\dagger R$ with high probability.
\end{thmnum}
Our method of proving Theorem \ref{THM:IntroB} allows for low sampling complexity (at least $k$ rows and columns must be sampled to achieve a valid CUR decomposition) and also allows columns and rows to be sampled independently of each other.  Moreover, our proof technique allows us to demonstrate stability of this sampling method in the following sense.

\begin{thmnum}[Theorem \ref{THM:StableCUR}]\label{THM:IntroC}
If $A$ has rank $k$, and $p_i,q_i$ are probability distributions determined by the column and row lengths of $A$, respectively, then for any probability distributions which satisfy $\tilde{p}_i\geq\alpha_ip_i$, $\tilde{q}_i\geq\beta_iq_i$ for some $\alpha_i,\beta_i>0$, sampling $O(k\log k)$ columns and rows of $A$ independently with replacement according to $\tilde{p}$ and $\tilde{q}$, respectively, implies that $A=CU^\dagger R$ with high probability.
\end{thmnum}

As a corollary, we find that uniform sampling of rows and columns yields an exact CUR decomposition with high probability; this result is new: the only previous results for uniform sampling were given by Chiu and Demanet under coherence assumptions on the matrix \cite{DemanetWu}.  
Additionally, we may combine Theorems \ref{THM:IntroC} and \ref{THM:IntroB}. Suppose that $\tilde{A} = A+E$ where $A$ has rank $k$, and we sample columns and rows of $\tilde{A}$ to form $\tilde{C},\tilde{U},\tilde{R}$.  These may be written as $\tilde{C} = C+E(:,J)$, for instance, where $C,U,$ and $R$ are the corresponding column, row, and intersections submatrices of the low rank matrix $A$.  It is natural to ask what the likelihood is that sampling from the noisy version of $A$ yields a CUR decomposition of $A$ itself. 

\begin{cornum}[Corollary \ref{COR:Uniform}]\label{COR:IntroA}
Suppose that $\tilde{A}=A+E$, with $A$ having rank $k$.  Then sampling $O(k\log k)$ columns and rows of $\tilde{A}$ uniformly with replacement yields $\tilde{C},\tilde{U},\tilde{R}$ such that $A=CU^\dagger R$ with high probability.
\end{cornum}

Something more general than Corollary \ref{COR:IntroA} may be said: indeed if the noise is small compared to the matrix $A$, then sampling $\tilde{A}$ according to its row and column lengths yields the same conclusion that $A=CU^\dagger R$ above.  See Remark \ref{REM:Atilde} for more details.

\subsection{Layout}

The rest of the paper develops as follows: Section \ref{SEC:Notation} establishes the notation used throughout the sequel; Section \ref{SEC:Viewpoints} derives both the exact CUR decomposition and the CUR approximation, which are the primary viewpoints given in the literature.  These methods are compared and shown to be equivalent in the exact decomposition case, and their difference is illustrated in the approximation case.  We also give a history of the decomposition and its influences in Section \ref{SUBSEC:History}.  Continuing on, Section \ref{SEC:Perturbation} contains the novel perturbation analysis for CUR approximations of noisy observations of low rank matrices, and the proof of Theorem \ref{THM:IntroA} is given along with tools to estimate the error bounds.  Section \ref{SEC:ColumnSelection} provides theoretical guarantees for obtaining an exact CUR decomposition of a noiseless low rank matrix and proves Theorems \ref{THM:IntroB}, \ref{THM:IntroC}, and Corollary \ref{COR:IntroA}.  Additionally, this section contains a survey of the different ways of forming CUR approximations in the literature and the corresponding column and row sampling methods utilized.  Section \ref{SEC:Numerics} contains some numerical experiments to illustrate some of the phenomena and theoretical results established beforehand, and Section \ref{SEC:Rank} illustrates a simple rank-estimation algorithm derived from our analysis.  Finally, the Appendix holds the proof of one of the intermediary theorems in Section \ref{SEC:ColumnSelection}.

\section{Notations}\label{SEC:Notation}

Here, the symbol $\K$ will represent either the real or complex field, and we will denote by $[n]$, the set of integers $\{1,\dots,n\}$ for convenience.

In the sequel, we will often have occasion to speak of submatrices of a matrix $A\in\K^{m\times n}$ with respect to certain columns and rows.  For this, we will use the Matlab-friendly notation $A(I,:)$ to denote the $|I|\times n$ row submatrix of $A$ consisting only of those rows of $A$ indexed by $I\subset[m]$, and likewise $A(:,J)$ will denote the $m\times|J|$ column submatrix of $A$ consisting only of those columns of $A$ indexed by $J\subset[n]$.  Therefore, $A(I,J)$ will be the $|I|\times|J|$ submatrix of those entries $a_{ij}$ of $A$ for which $(i,j)\in I\times J$.

The Singular Value Decomposition (SVD) of a matrix $A$ will typically be denoted by $A=W_A\Sigma_AV_A^*$ with the use of $W$ being preferred to the typical usage of $U$ since the latter will stand for the middle matrix in the CUR decomposition.  The truncated SVD of order $k$ of a matrix $A$ will be denoted by $A_k=W_k\Sigma_kV_k^*$, where $W_k$ comprises the first $k$ left singular vectors, $\Sigma_k$ is a $k\times k$ matrix containing the largest $k$ singular values, and $V_k$ comprises the first $k$ right singular vectors.  We will always assume that the singular values are positioned in descending order, and label them $\sigma_1\geq\sigma_2\geq\dots\geq\sigma_r\geq0$. To specify the matrix involved, we may also write $\sigma_i(A)$ for the $i$--th singular value of $A$.

Given a matrix $A\in\K^{m\times n}$, its Moore--Penrose pseudoinverse will be denoted by $A^\dagger\in\K^{n\times m}$.  We recall for the reader that this pseudoinverse is unique and satisfies the following properties: (i) $AA^\dagger A = A$, (ii) $A^\dagger AA^\dagger = A^\dagger$, and (iii) $AA^\dagger$ and $A^\dagger A$ are Hermitian.  Additionally, given the SVD of $A$ as above we have a simple expression for its pseudoinverse as $A^\dagger=V_A\Sigma_A^\dagger W_A^*$, where $\Sigma^\dagger$ is the $n\times m$ matrix with diagonal entries $\frac{1}{\sigma_i(A)}$, $i=1,\dots,r=\rank(A)$. 

Our analysis will consider a variety of matrix norms.  Some of the most important are the spectral norm $\|A\|_2:=\sup\{\|Ax\|_2:x\in S_{\K^n}\}$, where $S_{\K^n}$ is the unit sphere of $\K^n$ (in the Euclidean norm).  It is a useful fact that $\|A\|_2 = \sigma_1(A)$.  The Frobenius norm is another common norm that will be used, and has an entrywise definition, but also can be represented using the singular values as well, to wit
\[ \|A\|_F := \left(\sum_{i,j}a_{i,j}^2\right)^\frac12 = \left(\sum_i \sigma_i(A)^2\right)^\frac12.\]

A slight (but standard) abuse of terminology will be used in that we will call a matrix norm $\|\cdot\|$ submultiplicative provided $\|AB\|\leq\|A\|\|B\|$ without referencing the fact that $A$ and $B$ may be of different sizes and hence the norms are on different spaces of matrices.  Thus any submultiplicative norm for us will be one which is well-defined on $\K^{m\times n}$ for any $m$ and $n$ and is compatible in the manner prescribed.  Common examples are the spectral and Frobenius norm above, any induced matrix norm, and any Schatten $p$--norm, $1\leq p\leq\infty$ (which are defined by the right-most term in the Frobenius norm expression above but with the $\ell_2$ norm of the singular values replaced with an $\ell_p$ norm). Additionally, we say a matrix norm is unitarily invariant if $\|UAV\| = \|A\|$ for any unitary matrices $U$ and $V$.  Schatten $p$--norms (including the Frobenius norm) and the spectral norm are unitarily invariant.

We will utilize the following definition of Stewart \cite{Stewart_1977}: a family of unitarily invariant norms from $\bigcup_{m,n=1}^\infty\K^{m\times n}\to\R$ is called \textit{normalized} if $\|x\|=\|x\|_2$ for any vector $x$ considered as a matrix, and \textit{uniformly generated} if $\|A\|$ can be written as $\phi(\sigma_1(A),\dots,\sigma_n(A),0,\dots)$ for some symmetric function $\phi$.  Evidently all Schatten $p$--norms satisfy these conditions, and we have $\|\cdot\|_2\leq\|\cdot\|$ for any such norm.  In the sequel, $\|\cdot\|$ will always denote a normalized, uniformly generated, submultiplicative, unitarily invariant matrix norm unless otherwise specified.

Finally, we will use $\mathcal{N}(A)$ and $\mathcal{R}(A)$ to denote the nullspace and range of $A$, respectively; the symbol $f\gtrsim g$ will mean that $f\geq cg$ for some universal constant $c$, and for vectors $x,y\in\K^n$, $x\otimes y = xy^*$.

\section{Two viewpoints on CUR}\label{SEC:Viewpoints}

There appear to be two distinct starting points in most of the literature revolving around the CUR decomposition.  The first is as an exact matrix \textit{decomposition}, or \textit{factorization}.   

The second starting point is to find a \textit{low-rank approximation} to a given matrix.  In this modern era of large-scale, high-dimensional data analysis, dimension reduction techniques are absolutely crucial to leveraging data to make accurate conclusions about the world around us.  Despite significant advances in computational power, the data that we collect today is rarely amenable to fast computations without some sort of dimension reduction beforehand.  Moreover, data often has an intrinsically low-dimensional structure, which may be elucidated by dimension reduction techniques.

In the rest of this section, we will illustrate the conclusion of a practitioner of CUR based on each launching point mentioned here, and discuss why each point of view is useful.  Then we show that in a certain case (i.e. when an exact decomposition is obtained) these vantage points provide the exact same answer (i.e. the decomposition and approximation are one and the same).  Finally, we also discuss how the two conclusions are distinct from each other in the case one desires to use CUR as a low-rank approximation.  For ease of reading, we will withhold citations and historical notes about the theorems presented in this section until Section \ref{SUBSEC:History}, and simply alert the reader here that most of the results are known in some fashion, but it is our aim to provide some context and comparison of them here.

\subsection{Exact Decomposition: $A = CU^\dagger R$}

For our first tale regarding CUR, we begin by asking the question: given a matrix $A\in\K^{m\times n}$ with rank $k<\min\{m,n\}$, can we decompose it into terms involving only some of its columns and some of its rows.  Particularly, if we choose $k$ columns of $A$ which span the column space of $A$ and $k$ rows which span the row space of $A$, then we should be able to stitch these linear maps together to get $A$ itself back.   The answer, as it turns out, is yes as the following theorem shows (the intuition of the previous statement will be demonstrated in a subsequent figure).

\begin{theorem}\label{THM:CUR}
Let $A\in\K^{m\times n}$ have rank $k$, and let $I\subset[m]$ and $J\subset[n]$ with $|I|=t\geq k$ and $|J|=s\geq k$.  Let $C=A(:,J)$, $R=A(I,:)$, and $U=A(I,J)$.  If $\rank(U)=\rank(A)$, then 
\[ A = CU^\dagger R.\]
\end{theorem}

\begin{proof}
Given the constraint on $U$, it follows that $\rank(C)=\rank(R)=\rank(A)$.  Thus, there exists an $X\in\K^{s\times n}$ such that $A=CX$ (in fact there are infinitely many if $s>\rank(A)$).  Now suppose that $P_I$ is the row-selection matrix which picks out the rows of $A$ according to the index set $I$.  That is, we have $P_IA = R$, and likewise $P_IC = U$.  Then the following holds:
\begin{displaymath}
\begin{array}{lll}
A = CX & \Leftrightarrow & P_IA = P_ICX\\
& \Leftrightarrow & R = UX.\\
\end{array}
\end{displaymath}
The second equivalence is by definition of $U, R$, and $P_I$, while the forward direction of the first equivalence is obvious.  The backward direction holds by the assumption on the rank of $U$, and hence on $C$ and $R$.  That is, the row-selection matrix $P_I$ eliminates rows that are linearly dependent on the rest, and so any solution to $P_IA = P_ICX$ must be a solution to $A=CX$.  Finally, it suffices to show that $X = U^\dagger R$ is a solution to $R=UX$.  By the same argument as above, it suffices to show that $U^\dagger R$ is a solution to $RP_J=U = UXP_J$ if $P_J$ is a column-selection matrix which picks out columns according to the index set $J$.  Thus, noting that $UU^\dagger RP_J = UU^\dagger U = U$ completes the proof.
\end{proof}

\begin{figure}[ht]
\includegraphics[scale=0.4]{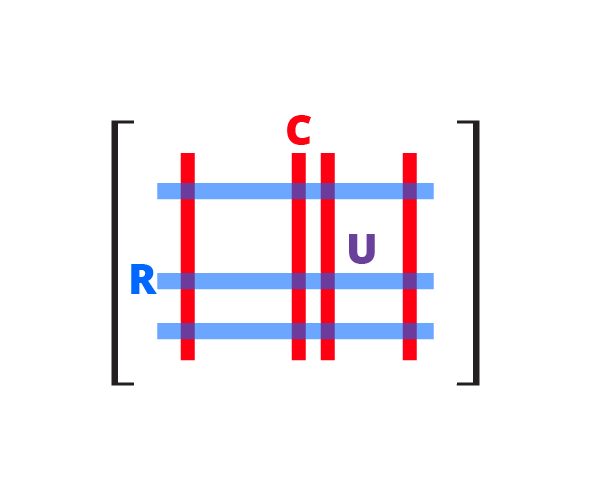}
\caption{Illustration of the CUR decomposition. $C$ is the red column submatrix of $A$, while $R$ is the blue row submatrix of $A$, and $U$ is their intersection (naturally purple).}\label{FIG:CUR}
\end{figure}

Figure \ref{FIG:CUR} provides an illustration of the CUR decomposition from a matrix point of view, whereas Figure \ref{FIG:CURSpace2} shows the intuition of the decomposition based on the viewpoint of the linear operators that the matrices represent.

\begin{remark}\label{REM:CUR}
Upon careful examination of the proof of Theorem \ref{THM:CUR}, we observe that the conclusion $A=CU^\dagger R$ holds also in the event that columns and rows of $A$ are repeated.  That is, if $I=\{i_1,\dots,i_t\}$, $i_k\in[n]$ and $J=\{j_1,\dots,j_s\}$, $j_k\in[m]$ (where the $i_k$ and $j_k$ are not necessarily distinct), and $C$ consists of columns $A(:,i_k)$, $k\in[t]$ and $R$ consists of rows $A(j_k,:)$, $k\in[s]$, with $U_{k,\ell}=A_{i_k,j_\ell}$, then provided $\rank(U)=\rank(A)$, we have $A=CU^\dagger R$.  This observation will be used in Section \ref{SEC:ColumnSelection}.
\end{remark}
\begin{figure}[ht]
		\includegraphics[scale=0.7]{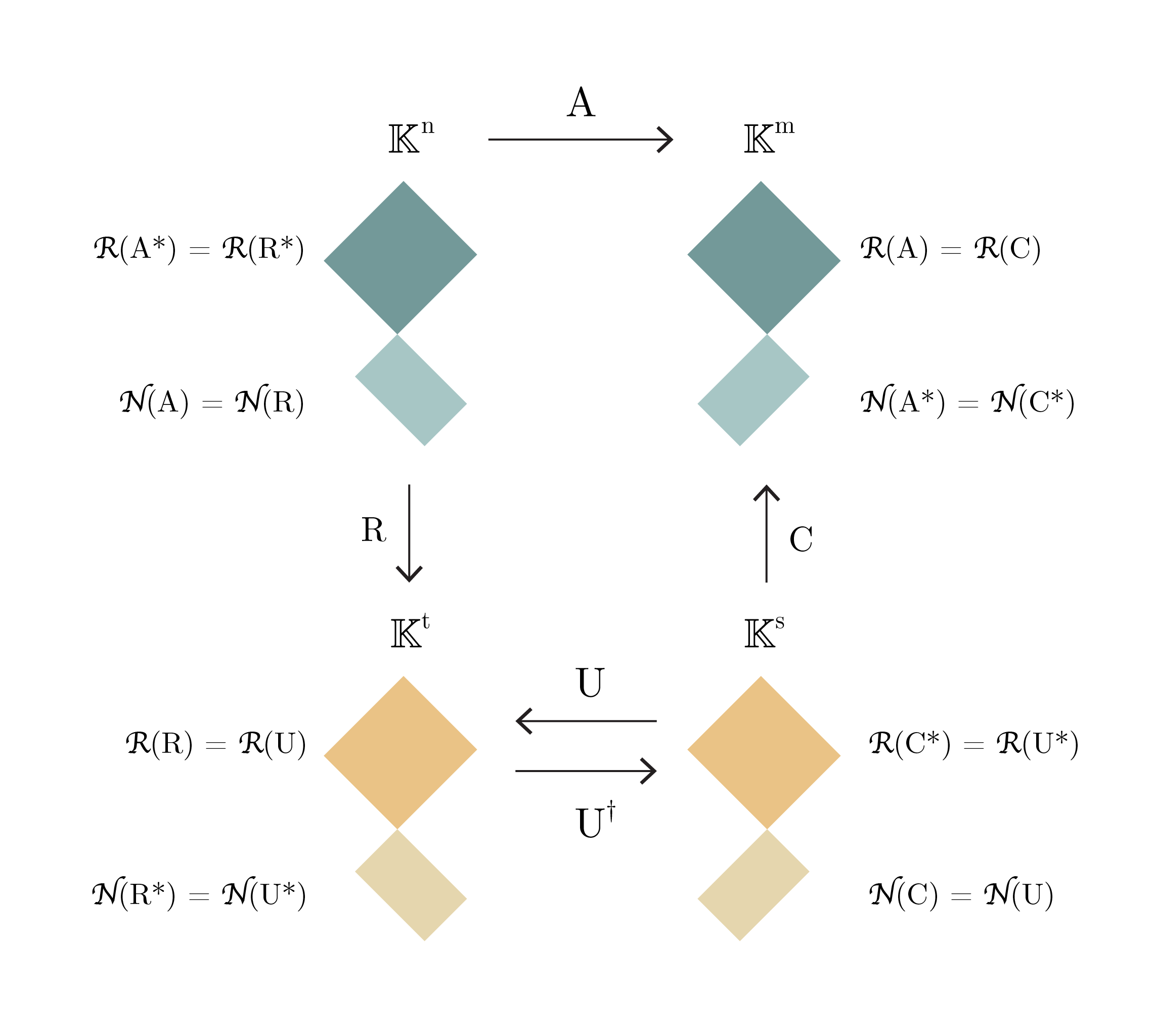}
		\caption{A diagram of the CUR decomposition viewed as linear operators. The illustration of the decomposition of the spaces is inspired by \cite{Strang}.  Because $R$ and $C$ capture the essential information of $A$, the Null spaces correspond as shown (see also Lemma \ref{LEM:Projections}), and $U^\dagger$ is the inverse of $U$ on its range, which allows the diagram to commute.}\label{FIG:CURSpace2}
\end{figure}

\subsection{Best Low Rank Approximation: $A \approx CC^\dagger AR^\dagger R$}

A low rank \textit{approximation} of a matrix $A$ is a matrix with small (compared to the size of the matrix) rank which is ideally close to $A$ in norm.  Many low rank approximation (and indeed decomposition) methods begin with the idea that perhaps a basis other than the canonical one is a ``better" basis in which to represent a given matrix, where the term better is vague and dependent upon the context.  The second viewpoint on CUR stems from this same idea and is, in our opinion, the one more closely tied to those interested in data science, whether in theory or practice.  So with that in mind, suppose that we would like to take a matrix $A\in\K^{m\times n}$, and find a rank $k$ approximation to it given some fixed $k$.  Of course, it is well known that if one desires the \textit{best} rank $k$ approximation to $A$, then one needs look no further than its truncated Singular Value Decomposition (SVD).  That is, if $U_k, \Sigma_k, V_k$ are the first $k$ left singular vectors, singular values, and right singular vectors, respectively, then we have
\begin{equation}\label{EQ:SVD Minimizer} \underset{X: \rank(X)=k}{\text{argmin}}\; \|A-X\|_{\xi} = U_k\Sigma_kV_k^*,\end{equation}
in the case that $\xi$ is any Schatten $p$--norm.

Given this observation, the reader might be forgiven for thinking, why should one look any further for a low-rank approximation for $A$ when the SVD provides the best?  Well, this is not the whole story, of course.  One reason to continue the search is that the SVD fails to preserve much of the structure of $A$ as a matrix.  For example, suppose that $A$ is a sparse matrix; then in general $U$ and $V$, hence $U_k$ and $V_k$, fail to be sparse.  Therefore, the SVD does not necessarily remain faithful to the matricial structure of $A$.  For yet another reason, we turn to the wisdom of Mahoney and Drineas \cite{DMPNAS}.  As they point out, a major factor in analyzing data is interpretability of the results.  Consequently, if one manipulates data in some fashion, one must do it in such a way as to still be able to make a meaningful conclusion.

Let us take the SVD as a case in point: suppose in a medical study, a researcher observes a large number $n$ of gene expression levels in $m$ patients and concatenates the data into an $m\times n$ matrix $A$. Supposing the desired outcome is to determine which genes are most indicative of cancer risk in patients, the researcher attempts to reduce the dimension of the data significantly, and so takes the truncated SVD of this matrix, and looks at the data in the $k$--dimensional basis $U_k$.  But what does a singular vector in $U_k$ correspond to?  It will generally be a linear combination of the genes; so what would it mean, say, that the first two singular vectors capture the majority of information in the data if the singular vectors are combinations of all of the gene expressions?  In using the SVD, interpretability of the data has been utterly lost.

Finally, computing the full SVD of a matrix $A$ is expensive (the na\"{i}ve direct algorithm requires $O(\min\{mn^2, nm^2\})$ operations).  Nonetheless, it is more economical to compute the truncated SVD; indeed, computing $A_k$ requires only $O(mnk)$ operations \cite{tropp}.

So what would be a better alternative?  It seems natural to ask: can we choose only a few representative columns of our data matrix $A$ such that they essentially capture all of the necessary information about $A$?  Or more precisely, can we project $A$ onto the space spanned by some representative columns such that the result is close to $A$ in norm?  Unsurprisingly, this problem is important enough to be named, and is typically called the \textit{Column Subset Selection Problem}.  Before stating the problem, consider the preliminary observation that if $C$ is a column submatrix of $A$, then the following holds for $\xi = 2$ or $F$ \cite[Theorem 10.B.7]{MOA_2011}:
\begin{equation}
\label{EQ:MinA-CX}\underset{X}{\argmin}\;\|A-CX\|_{\xi} = C^\dagger A.  \end{equation} 
This property follows from the fact that the Moore--Penrose pseudoinverse gives the least-squares solution to a system of equations.  With this observation in hand, combined with the knowledge that $CC^\dagger:\K^m\to\K^m$ is the projection operator onto the column space of $C$, the Column Subset Selection Problem may be stated as follows.

\begin{problem}[Column Subset Selection Problem]
Given a matrix $A\in\K^{m\times n}$ and a fixed $k\in[n]$, find a column submatrix $C=A(:,J)$ which solves the following:
\[ \underset{J\subset[n], |J|=k}{\underset{C=A(:,J)}{\min}}\; \|A-CC^\dagger A\|_\xi,\]
where $\xi$ is a norm allowed to be specified -- typically chosen to be either $2$ or F.
\end{problem}

The astute observer will notice that this problem is difficult in general (more on its complexity in Section \ref{SUBSEC:History}).  Indeed, there are $\binom{n}{k}$ choices of matrices $C$ over which to minimize.  However, let us set this difficulty aside for the moment and return to the problem at hand.  Equation \eqref{EQ:MinA-CX} tells us that given a column submatrix $C$ which solves the Column Selection Problem, $A$ is best represented by $A\approx CC^\dagger A$.  But why stop there?  We may as well also select some rows of $A$ which best capture the essential information of its row space.  Similar to \eqref{EQ:MinA-CX}, one may easily show that, given a row submatrix $R$ of $A$, the following holds for $\xi=2$ or $F$:
\[ \underset{X}\argmin\;\|A-XR\|_\xi = AR^\dagger.\]
Therefore, we now wish to find the best rows which minimize the argument above.  Rather than calling this the ``Row Subset Selection Problem," simply note that this is equivalent to solving the Column Subset Selection Problem on $A^*$.

It is now natural to stitch these tasks together, and attempt to find the minimizer of $\|A-CZR\|_\xi$ given a fixed $C$ and $R$.

\begin{proposition}[\cite{stewart_minimizer}]\label{PROP:CAR}
Let $A\in\K^{m\times n}$ and $C$ and $R$ be column and row submatrices of $A$, respectively.  Then the following holds:
\[ \underset{Z}{\text{argmin}}\;\|A-CZR\|_F = C^\dagger AR^\dagger.\]
\end{proposition}

Given the result of Proposition \ref{PROP:CAR}, much of the literature surrounding the CUR decomposition takes
\[ A\approx CC^\dagger AR^\dagger R\]
to be the CUR decomposition of $A$.  However, to be more precise, we will herein term this a CUR \textit{approximation} of $A$.

It should be noted that Proposition \ref{PROP:CAR} is not true for spectral norm as the follow example demonstrates.
\begin{example}
Consider \[A=\begin{bmatrix}
1& 1\\
1&2\\
\end{bmatrix}, C=\begin{bmatrix}
1\\
1\\
\end{bmatrix}, R=\begin{bmatrix}
1&1\\
\end{bmatrix}.\]  First note that in this case when evaluating $\text{argmin}_z\|A-CzR\|_\xi$, $z$ is simply a scalar. It is a simple exercise to demonstrate that $\|A-CzR\|_F^2 = 3(1-z)^2+(2-z)^2$, which is minimized whenever $z=\frac54 = C^\dagger AR^\dagger$. However, under the 2-norm, one can show that $z=1$ is the optimal solution by considering the maximal eigenvalue of $A-CzR$ and computing the minimizer explicitly. Thus Proposition \ref{PROP:CAR} does not hold for the spectral norm in general.
\end{example}

\begin{example}
Another example of a different sort is to take $A = I$, and $C$ and $R$ to again be the first column and row, respectively.  Then $C^\dagger AR^\dagger=1$, but the eigenvalues of $A-CzR$ are $1$ and $1-z$.  So any $z\in[0,2]$ yields a minimum value for $\|A-CzR\|_2$ of 1.  This example also illustrates the key fact that there may be a continuum of matrices $U$ for which $\|A-CUR\|_2$ is constant.
\end{example}

\subsection{Equality in the Exact Case}

Previously, we discussed two starting points and conclusions for what is termed the CUR decomposition in the literature.  The purpose of this and the next subsection is to compare these two viewpoints.  First, we demonstrate in Theorem \ref{THM:Characterization} that in the exact decomposition case, these viewpoints are in fact one and the same. However, in the process of doing so, we do more by giving several equivalent characterizations of when an exact CUR decomposition is obtained. Before stating this theorem, we make note of some useful facts about the matrices involved.

\begin{lemma}\label{LEM:Projections}
Suppose that $A, C, U,$ and $R$ are as in Theorem \ref{THM:CUR}, with $\rank(A)=\rank(U)$.  Then $\mathcal{N}(C)=\mathcal{N}(U)$, $\mathcal{N}(R^*)=\mathcal{N}(U^*)$, $\mathcal{N}(A)=\mathcal{N}(R)$, and $\mathcal{N}(A^*)=\mathcal{N}(C^*)$.  Moreover,
\[ C^\dagger C = U^\dagger U,\qquad  RR^\dagger=UU^\dagger,\]
\[ AA^\dagger = CC^\dagger, \quad \text{and} \quad A^\dagger A = R^\dagger R.\]
\end{lemma}

\begin{proof}
As noted in the proof of Theorem \ref{THM:CUR}, the constraint that $\rank(U)=\rank(A)=k$ implies also that $C$ and $R$ have rank $k$ as well.  The statements about the kernels follows directly from this observation.  To prove the moreover statements, notice that if $C\in\K^{m\times s}$, then $C^\dagger C:\K^s\to\K^s$ is the orthogonal projection onto $\mathcal{N}(C)^\perp$.  However, $\mathcal{N}(C)=\mathcal{N}(U)$ since their ranks are the same and $U$ is obtained by selecting certain rows of $C$, and hence $\mathcal{N}(C)^\perp=\mathcal{N}(U)^\perp$.  Therefore, $C^\dagger C$ is the orthogonal projection onto $\mathcal{N}(U)^\perp$, but this is $U^\dagger U$.  Similarly, $RR^\dagger$ is the projection onto $\mathcal{N}(R^*)^\perp = \mathcal{N}(U^*)^\perp$, whence $RR^\dagger=UU^\dagger$, and the proof is complete.  The final two statements follow by the same reasoning, so the details are omitted.
\end{proof}

\begin{lemma}\label{rankE}
	Let $A\in\mathbb{K}^{m\times n}$ and $B\in\mathbb{K}^{n\times p}$. If $\rank(A)=\rank(B)=n$, then $\rank(AB)=n$.
	\end{lemma}
	The proof of Lemma \ref{rankE} is a straightforward exercise using Sylvester's rank inequality, and so is omitted.
	
	\begin{corollary}\label{CR_RANK_U}
	   Let $A\in\mathbb{K}^{m\times n}$ with $\rank(A)=k$. Let $C=A(:,J)$ and $R=A(I,:)$ with $\rank(C)=\rank(R)=k$. Then $\rank(U)=k$ with $U=A(I,J)$.
	\end{corollary}
	\begin{proof}
	Assume that $A$ has truncated SVD $A=W_k\Sigma_k V_k^*$. Then $C=W_k\Sigma_k (V_k(J,:))^*$. Since $\rank(C)=k$, we have $\rank(V_k(J,:))=k$. Similarly, we can conclude that $\rank(W_k(I,:))=k$. Note that $U=A(I,J)=W_k(I,:)\Sigma_k (V_k(J,:))^*$, hence by Lemma \ref{rankE}, we have $\rank(U)=k$.
	\end{proof}
	
As a side note, one can also demonstrate a different form for $U$ aside from simply intersection of $C$ and  $R$.

\begin{proposition}\label{PROP:U=RAC}
	Suppose that $A$, $C$, $U$, and $R$  are as in Theorem \ref{THM:CUR} (but without any assumption on the rank of $U$). Then \[U=RA^{\dagger}C.\]
\end{proposition}
\begin{proof}
Let the full singular value decomposition of $A$ be $A=W_A\Sigma_A V_A^*$. Then $C=A(:,J)=W_A\Sigma_AV_A^*(:,J)$, $R=A(I,:)=W_A(I,:)\Sigma_AV_A^*$, and $U=U_A(I,:)\Sigma_AV_A^*(:,J)$.
Therefore, we have
\begin{eqnarray*}
	RA^{\dagger}C&=&W_A(I,:)\Sigma_AV_A^*V_A\Sigma_A^\dagger W_A^*W_A\Sigma_AV_A^*(:,J)\\
	&=&W_A(I,:)\Sigma_A\Sigma_A^\dagger\Sigma_AV_A^*(:,J)\\
	&=&W_A(I,:)\Sigma_AV_A^*(:,J)=U.
\end{eqnarray*}
\end{proof}	
	
We are now in a position to state our main theorem characterizing exact CUR decompositions.

\begin{theorem}\label{THM:Characterization}
Let $A\in\K^{m\times n}$ and $I\in[m]$, $J\in[n]$ (possibly having redundant entries).  Let $C=A(I,:)$, $U=A(I,J)$, and $R=A(:,J)$.  Then the following are equivalent:
\begin{enumerate}[(i)]
    \item\label{ITEM:Rank} $\rank(U)=\rank(A)$
    \item\label{ITEM:CUR} $A=CU^\dagger R$
    \item\label{ITEM:ACCARR} $A = CC^\dagger AR^\dagger R$
    \item\label{ITEM:Adagger} $A^\dagger = R^\dagger UC^\dagger$
    \item\label{ITEM:Spans} $\rank(C)=\rank(R)=\rank(A)$.
\end{enumerate}
\end{theorem}
\begin{proof}
Remark \ref{REM:CUR} is the implication $(\ref{ITEM:Rank})\Rightarrow(\ref{ITEM:CUR})$.  To see $(\ref{ITEM:CUR})\Rightarrow(\ref{ITEM:Rank})$,  suppose to the contrary that $\rank(U)\neq\rank(A)$.  By construction of $U$, this implies that $\rank(U)<\rank(A)$.  On the other hand,
\[\rank(A)=\rank(CU^\dagger R)\leq\rank(U^\dagger)=\rank(U)<\rank(A), \]
which yields a contradiction.  Hence $(\ref{ITEM:Rank})$ holds.

The forward direction of $(\ref{ITEM:Rank})\Leftrightarrow(\ref{ITEM:Spans})$ is easily seen, while the reverse direction is the content of Corollary \ref{CR_RANK_U}, and $(\ref{ITEM:Spans})\Rightarrow(\ref{ITEM:ACCARR})$ is obvious given that under the assumption on the spans, $AA^\dagger = CC^\dagger$ according to Lemma \ref{LEM:Projections}.  To see $(\ref{ITEM:ACCARR})\Rightarrow(\ref{ITEM:Spans})$, recall that by construction, $\sspan(C)\subset\sspan(A)$, but $(\ref{ITEM:ACCARR})$ implies that $\sspan(A)\subset\sspan(C)$.  A similar argument shows that $\sspan(R^*)=\sspan(A^*)$, which yields $(\ref{ITEM:Spans})$.

Now suppose that $(\ref{ITEM:Rank})$ (and equivalently ($\ref{ITEM:CUR}$)) holds.  By Proposition \ref{PROP:U=RAC}, $U=RA^\dagger C$, hence the following holds:
\begin{align*}
    R^\dagger UC^\dagger &= R^\dagger RA^\dagger CC^\dagger\\
    &= A^\dagger AA^\dagger AA^\dagger\\
    & = A^\dagger,
\end{align*}
where the second equality follows from Lemma \ref{LEM:Projections} (which requires the assumption $(\ref{ITEM:Rank})$) and the last from properties defining the Moore--Penrose pseudoinverse.  Thus $(\ref{ITEM:Rank})\Rightarrow(\ref{ITEM:Adagger})$. 
Conversely, suppose that $A^\dagger = R^\dagger U C^\dagger$.  Then by Proposition \ref{PROP:U=RAC},
\[ A = AA^\dagger A = AR^\dagger UC^\dagger A = AR^\dagger RA^\dagger CC^\dagger A.\]
Hence $\rank(A)=\rank(AR^\dagger RA^\dagger CC^\dagger A)\leq\rank(C)\leq\rank(A)$. Thus, $\rank(A)=\rank(C)$. Similarly, $\rank(A)=\rank(R)$. Thus an appeal to Corollary \ref{CR_RANK_U} completes the proof of $(\ref{ITEM:Adagger})\Rightarrow(\ref{ITEM:Rank})$.
\end{proof}

Theorem \ref{THM:Characterization} provides several novel characterizations of exact CUR decompositions, and in particular, the equivalence of conditions (\ref{ITEM:CUR}) and $(\ref{ITEM:ACCARR})$ demonstrates that the two proposed viewpoints indeed match in the exact decomposition case (and only in this case).

We now turn to some auxiliary observations.

\begin{proposition}\label{PROP:Udagger}
Suppose that $A, C, U,$ and $R$ are as in Theorem \ref{THM:CUR}, with $\rank(A)=\rank(U)$.  Then \[ U^\dagger=C^\dagger A R^\dagger.\]
\end{proposition}

\begin{proof}
Under these assumptions, $A=CU^\dagger R$ by Theorem \ref{THM:CUR}.  On account of Lemma \ref{LEM:Projections},
\begin{align*}
C^\dagger AR^\dagger & = C^\dagger CU^\dagger  R R^\dagger \\
& = U^\dagger UU^\dagger UU^\dagger\\
& = U^\dagger,
\end{align*}
where the final step follows from basic properties of the Moore--Penrose pseudoinverse.
\end{proof}

\begin{remark}
The condition in Proposition \ref{PROP:Udagger} is not sufficient; if $U^\dagger=C^\dagger A R^\dagger$, then $\rank(A)$ may not equal to $\rank(U)$. For example, let
\[A=\begin{bmatrix}
0&I\\
I&0
\end{bmatrix},\] and let $C=\begin{bmatrix}
0\\
I
\end{bmatrix}$ and $R=\begin{bmatrix}
0&I
\end{bmatrix}$. Then $U=U^\dagger=0$, and $C^\dagger A R^\dagger=\begin{bmatrix}
0 &I
\end{bmatrix}\cdot \begin{bmatrix}
0&I\\
I&0
\end{bmatrix}\cdot \begin{bmatrix}
0\\
I
\end{bmatrix}=0$. Thus $U^\dagger=C^\dagger AR^\dagger$, but $\rank(A)\neq \rank(U)$.
\end{remark}

Note that Proposition \ref{PROP:U=RAC} holds for any choice of column and row submatrices $C$ and $R$.  However, Proposition \ref{PROP:Udagger} requires the additional assumption that $U$ and $A$ have the same rank.  Recall that Proposition \ref{PROP:Udagger} does not follow immediately from Proposition \ref{PROP:U=RAC} without this additional assumption given the fact that $(AB)^\dagger$ is not $B^\dagger A^\dagger$ in general.  Additionally, since the conclusion of Proposition \ref{PROP:Udagger} does not imply $\rank(U)=\rank(A)$, then it cannot imply any of the equivalent conditions given in Theorem \ref{THM:Characterization} in general.  We end this section with an interesting question related to this proposition.  Evidently, if $U^\dagger=C^\dagger AR^\dagger$, then $CU^\dagger R = CC^\dagger AR^\dagger R$.  Does the converse always hold? Note that if $\rank(U)=\rank(A)$, then the converse is true by Theorem \ref{THM:Characterization} because both quantities are $A$.  However, to show that it holds in general, one must determine if it holds in the case that $A\neq CU^\dagger R$.  At the moment, we leave this as an open question; however, numerical experiments give evidence that it is possibly true.

\subsection{Distinctness in the Approximation Case}

Now let us give a simple example to show that in the CUR approximation case, the choice of $C^\dagger AR^\dagger$ for the middle matrix in the CUR approximation indeed gives a better approximation than using $U^\dagger$, and in fact these matrices are not the same in this case.

\begin{example}
Consider the full rank matrix \[
A = \begin{bmatrix}1 & 2\\ 3 & 4\\
\end{bmatrix},
\]
and let \[C = A(:,1) = \begin{bmatrix}1\\3\\
\end{bmatrix}, \text{ and } R = A(1,:) = \begin{bmatrix}1 & 2\\ \end{bmatrix}.\]

Then $U=U^\dagger = 1$, and notice that $\|A-CU^\dagger R\|_F = 2$.  However, minimizing the function $f(z) = \|A-CzR\|_F$ yields a minimal value of $1.0583$ at $z=0.76$, which is $C^\dagger AR^\dagger$.  This simple example demonstrates that in the approximation case when fewer rows and columns are chosen than the rank of $A$, the matrices $U^\dagger$ and $C^\dagger AR^\dagger$ are different in general.
\end{example}

\subsection{Storage and Complexity Considerations}

It is pertinent to discuss the storage cost and complexity of the approximations described above. As these are easily computed, we tabulate them here for the reader without proof for aesthetic purposes.  We assume here that the matrix rank is $k$, and that $r\geq k$ rows and $c\geq k$ columns are chosen as in Theorem \ref{THM:CUR}; all complexity and storage values in Table \ref{TAB:ComplexityStorage} are to be taken as $O(\cdot)$.

\begin{table}[h!]
\begin{tabular}{|c||c|c|}\hline
     Method & Complexity & Storage  \\ \hline\hline
     Full SVD & $\min\{m^2n,mn^2\}$ & $m^2+n^2+k$\\
     Truncated SVD & $mnk$ & $k(m+n+1)$\\
     $CU^\dagger R$ & $\min\{cr^2,c^2r\}$ & $mc+nr+k^2$ \\ 
    $CC^\dagger ARR^\dagger$ & $mc^2 + nr^2$  & $mc+nr+mn$    \\ \hline
\end{tabular}\caption{Complexity and storage sizes of different approximations.}\label{TAB:ComplexityStorage}
\end{table}
Note that in the case $c=r=k$, the storage of the truncated SVD and CUR decomposition are essentially the same, differing by a factor of $k(k-1)$, and the latter's complexity becomes $O(k^3)$, which is smaller than even the truncated SVD.

\subsection{A History}\label{SUBSEC:History}

A precise history of the CUR decomposition (Theorem \ref{THM:CUR}) is somewhat elusive and its origins appear to be folklore at this point.  Many papers cite Gantmacher's book \cite{Gantmacher} without providing a specific location, but noting that the term \textit{matrix skeleton} is used therein.  The authors could not verify this source, as a search of the term \textit{skeleton} in a digital copy of the book turned up no results.  However, we find it implicitly in a paper by Penrose from 1956 \cite{Penrose56} (this is a follow-up paper to the one defining the pseudoinverse that now bears his name).  Therein, Penrose notes (albeit without proof) that any matrix may be written (possibly after rearrangement of rows and columns) in the form
\[ A = \begin{bmatrix} 
B & D\\ E & EB^{-1}D
\end{bmatrix},\]
where $B$ is any nonsingular submatrix of $A$ with $\rank(B)=\rank(A)$, whereupon one can immediately get a valid CUR decomposition for this matrix by choosing $R = \begin{bmatrix}B & D\end{bmatrix}$ and $C =\begin{bmatrix} B \\ E\end{bmatrix}$.  Subsequently, Theorem \ref{THM:CUR} is stated without proof in the case that $U$ is square and invertible in \cite{Goreinov}.  The proof for square submatrices $U$ that are not full rank appears in \cite{CaiafaTensorCUR}.  To the authors' knowledge, the first time the direct proof of the general rectangular $U$ case appears in the literature was in \cite{AHKS}; the proof given here is essentially the one given therein.

The trail of the CUR decomposition as a computational tool runs cold for some time after Penrose's paper, but may be picked up again in the works of Goreinov, Zamarashkin, and Tyrtyshnikov \cite{Goreinov,Goreinov2,Goreinov3}.  The authors therein take for granted the exact CUR decomposition of the form $A=CU^{-1}R$, which is a special case of Theorem \ref{THM:CUR} whenever exactly $\rank(A)$ rows and columns are chosen to form $U$.  From this launching point, they ask the question: if $A$ is approximately low rank, then how can one obtain a good CUR approximation to $A$ in the spectral norm?  However, their analysis is for general matrices $U$ rather than simply being of the form $U=A(I,J)$.  They provide precise estimates on CUR approximations in terms of a related min-max quantity.  These estimates are universal in the sense that the derived upper bound is not dependent upon the given matrix $A$.

The works of Goreinov, Zamarashkin, and Tyrtyshnikov are perhaps the modern starting point of CUR approximations, and have since sparked a significant amount of activity in the area.  Specifically, Drineas, Kannan, and Mahoney have considered a large variety of CUR approximations inspired by the analysis of \cite{Goreinov}.  They again admit flexibility in the choice of the matrix $U$, and prove many relative and additive error bounds for their approximations, as well as determining algorithms for computing the approximations which are computationally cheap,  \cite{Bien,DKMIII,DM05,DMM08,DMPNAS,VoroninMartinsson}.  We leave the discussion of their exact approximations to the survey in Section \ref{SEC:LiteratureSurvey}, but note here that a typical result in these works quantifies how well a CUR approximation with randomly oversampled columns and rows approximates the truncated SVD up to some penalty.  Similar work has been done by Chiu and Demanet \cite{DemanetWu} for uniform sampling of columns and rows with an additional coherency assumption on the given matrix $A$.  These works will be discussed in more detail in the sequel.

Applications of CUR approximations have become prevalent, including works on astronomical object detection \cite{Yip14}, mass spectrometry imaging \cite{Yang15}, matrix completion \cite{Xu15}, the joint learning problem \cite{CURTPAMI}, and subspace clustering \cite{AHKS}.

A very general framework for randomized low-rank matrix approximations is given in the excellent work of Halko, Martinsson, and Tropp \cite{tropp}, wherein they discuss CUR approximations as well as related methods such as interpolative decompositions and randomized approximations to the SVD.

The Column Subset Selection Problem (CSSP) has been well-studied in the theoretical computer science and randomized linear algebra literature \cite{AltschulerGreedyCSSP, Boutsidis2009, Deshpande2010, LiDeterministicCSSP, OrdozgoitiCSSP, tropp2009column, Yang}.  Indeed as a dimension reduction tool for data analysis, the CSSP is completely natural.  Many such methods attempt to represent given data in terms of a basis of reduced dimension which capture the essential information of the data.  Whereas Principal Component Analysis may result in a loss of interpretability as mentioned before, column selection corresponds to choosing actual columns of the data, and hence the minimization in the CSSP is attempting to find the best features that capture the most information of the data.  Feature selection as a preconditioner to task-based machine learning algorithms -- e.g. neural network or support vector machine classifiers -- is a critical step in many applications, and thus a thorough understanding of the CSSP is important for the analysis of data.  Column selection has also been applied in drawing large graphs \cite{KhouryScheidegger}.

As for complexity, the CSSP is believed to be NP--hard, with a purported proof given by Shitov \cite{Shitov}; a proof of its UG--hardness was given already by \c{C}ivril \cite{Civril}.  UG--hardness is a relaxed notion which states that a problem is NP--hard assuming the \textit{Unique Games Conjecture} (see \cite{Khot} for a formal description).

\section{Perturbation Analysis for CUR Approximations}\label{SEC:Perturbation}

We now turn to a perturbation analysis suggested by the CUR approximations described above.  Our primary task will be to consider matrices of the form \[\tilde{A} = A+E,\] where $A$ has low rank $k$, and $E$ is a (generally) full rank noise matrix.  For experimentation in the sequel we will consider $E$ to be a random matrix drawn from a certain distribution, but here we do not make any assumption on its entries.  We are principally interested in the case that $E$ is ``small" in a suitable sense, and so the observed matrix $\tilde{A}$ is really a small perturbation of the low rank matrix $A$.  To this end, most of our analysis will contain upper bounds on a CUR approximation of $\tilde{A}$ in terms of a norm of the noise $E$.

Since we now well-understand how a CUR \textit{decomposition} of $A$ behaves, we would like to utilize this understanding to tell us something about how a CUR \textit{approximation} of $\tilde{A}$ behaves.  To set some notation, if $\tilde{A}=A+E$, and we consider $\tilde{C}=\tilde{A}(:,J)$, $\tilde{R}=\tilde{A}(I,:)$, and $\tilde{U}=\tilde{A}(I,J)$ for some index sets $I$ and $J$, then we write 
\begin{equation}\label{EQ:tildes}
\tilde{C} = A(:,J)+E(:,J) =: C+E(:,J),\quad \tilde{R} = R+E(I,:),\quad \tilde{U} = U+E(I,J),
\end{equation}
where $R:=A(I,:)$ and $U:=A(I,J)$.
Thus if we choose columns and rows, $\tilde{C}$ and $\tilde{R}$ of $\tilde{A}$, we would like to determine how this compares to the underlying approximation of the low rank matrix $A$ by its columns and rows, $C$ and $R$.
Essentially all of our results in the sequel will be of the form
\[ \|A-\tilde{C}\tilde{U}^\dagger\tilde{R}\|\leq \|A-CU^\dagger R\|+O(\|E\|).\]
The work then will be to determine what the $O(\|E\|)$ error term is and to estimate the likelihood that $\|A-CU^\dagger R\|=0$.

For ease of notation, we will use the conventions that $E_I:=E(I,:)$, $E_J:=E(:,J)$, and $E_{I,J}:=E(I,J)$; since $I$ and $J$ are always reserved for subsets of the rows and columns, respectively, we trust this will not cause confusion.

\subsection{Preliminaries from Matrix Perturbation Theory}

Before stating our results, we collect some useful facts from perturbation theory. The first is due to Weyl:

\begin{theorem}\cite[Corollary 8.6.2.]{GolubVanLoan} \label{THMHoffmanWielandt}
	If $B,E\in\K^{m\times n}$ and $\tilde{B}=B+E$, then for $1\leq j\leq \min\{m,n\}$,
	\begin{equation}
	\left|\sigma_j(B)-\sigma_j(\tilde{B})\right|\leq \sigma_1(E)=\|E\|_2.
	\end{equation}
\end{theorem}

Note that Theorem \ref{THMHoffmanWielandt} holds in greater generality and is due to Mirsky \cite{Mirsky}.  Therein, it was shown that for any normalized, uniformly generated, unitarily invariant norm $\|\cdot\|$, \[\|\textnormal{diag}(\sigma_1(B)-\sigma_1(\tilde{B}),\dots)\|\leq\|E\|.\]  We will have occasion to use this estimate in the sequel.

The following Theorem of Stewart provides an estimate for how large the difference of pseudo-inverses can be.  

\begin{theorem}\cite[Theorems 3.1--3.4]{Stewart_1977}\label{THM:Stewart}
	Let $\|\cdot\|$ be any normalized, uniformly generated, unitarily invariant norm on $\K^{m\times n}$.  For any $B,E\in\K^{m\times n}$ with $\tilde{B}=B+E$, if $\rank(\tilde{B})=\rank(B)$, then
	\[
	\|B^{\dagger}-\tilde{B}^{\dagger} \|\leq\mu \|\tilde{B}^{\dagger}\|_2\|B^{\dagger}\|_2\|E\|,
	\]
	where $1\leq\mu\leq 3$ is a constant depending only on the norm.
    
    If $\rank(\tilde{B})\neq\rank(B)$, then
    \[ \|B^\dagger-\tilde{B}^\dagger\|\leq\mu\max\{\|\tilde{B}^\dagger\|_2^2,\|B^\dagger\|_2^2\}\|E\| \text{ and } 1/\|E\|_2\leq\|B^\dagger-\tilde{B}^\dagger\|_2. \]
\end{theorem}

The precise value of $\mu$ depends on the norm used and the relation of the rank of the matrices to their size; in particular, $\mu=3$ for an arbitrary norm satisfying the hypotheses in Section \ref{SEC:Notation}, whereas $\mu=\sqrt{2}$ for the Frobenius norm, and $\mu=\frac{1+\sqrt{5}}{2}$ (the Golden Ratio) for the spectral norm.

The preceding theorems yield the following immediate corollary.

\begin{corollary}\label{COR:AdaggerBounds}
	With the assumptions of Theorem \ref{THM:Stewart}, if $\tilde{B}=B+E$ and $\rank(\tilde{B})=\rank(B)=k$, then
	\[|\|B^{\dagger}\|-\|\tilde{B}^{\dagger}\||\leq \mu\|B^\dagger\|_2\|\tilde{B}^{\dagger}\|_2\|E\|.\]
Moreover, if $\sigma_k(B)>\mu\|E\|$, then
	\[\frac{\|B^\dagger\|}{1+\mu\|B^\dagger\|_2\|E\|}\leq\|\tilde{B}^{\dagger}\|\leq \frac{\|B^\dagger\|}{1-\mu\|B^\dagger\|_2\|E\|}.\]
\end{corollary}

Regard that from the representation of $B^\dagger$ in terms of the SVD of $B$ mentioned in Section \ref{SEC:Notation}, we have $\|B^\dagger\|_2=1/\sigma_{\min}(B)$, where $\sigma_{\min}(B)$ is the smallest nonzero singular value of $B$; this is sometimes how the inequalities in Corollary \ref{COR:AdaggerBounds} are written.

\subsection{Perturbation Estimates for CUR Approximations}

To begin, let us consider the CUR approximation suggested by the exact decomposition of Theorem \ref{THM:CUR}.  The following proposition will be useful in estimating some of the terms that arise in the subsequent analysis.
 
 \begin{proposition}\label{PROP:NormTerms}
Suppose that $A, C, U$, and $R$ are as in Theorem \ref{THM:CUR} such that $A=CU^\dagger R$, and suppose that $\rank(A)=k$.  Let $A=W_k\Sigma_k V_k^*$ be the truncated SVD of $A$.  Then for any unitarily invariant norm $\|\cdot\|$ on $\K^{m\times n}$, we have
\[ \|CU^\dagger\|= \|W_{k,I}^\dagger\|,\quad \text{and} \quad \|U^\dagger R\|= \|V_{k,J}^\dagger\|,\]
where $W_{k,I}:=W_k(I,:)$ and $V_{k,J}:=V_k(J,:)$.
\end{proposition}

\begin{proof}
First, note that by Proposition \ref{PROP:Udagger} and the fact that $CC^\dagger=AA^\dagger$ (Lemma \ref{LEM:Projections}), we have
\[ CU^\dagger = CC^\dagger AR^\dagger = AA^\dagger AR^\dagger = AR^\dagger,\] and likewise
\[ U^\dagger R = C^\dagger A.\]
As noted in Proposition \ref{PROP:U=RAC}, we have that \[R = W_{k}(I,:)\Sigma_{k}V_{k}^*=:W_{k,I}\Sigma_{k}V_{k}^*.\] Consequently, \[AR^\dagger = W_k\Sigma_k V_k^*(W_{k,I}\Sigma_{k}V_{k}^*)^\dagger. \]
To estimate the norm, let us first notice that the pseudoinverse in question turns out to satisfy
\[(W_{k,I}\Sigma_{k}V_{k}^*)^\dagger = (V_{k}^*)^\dagger\Sigma_{k}^{-1} W_{k,I}^\dagger.\]
This is true on account of the fact that $W_{k,I}$ has full column rank, $V_{k}^*$ has orthonormal rows, and $\Sigma_{k}$ is invertible by assumption.  Next, note that since $V_k^*$ has orthonormal rows, $(V_k^*)^\dagger = V_k$.  Putting these observations together, we have that

\begin{align}\label{EQARNorm}
\|AR^\dagger\| & = \|W_k\Sigma_k V_k^*(W_{k,I}\Sigma_{k}V_{k}^*)^\dagger\| \nonumber\\
& = \|\Sigma_kV_k^*V_{k}\Sigma_{k}^{-1} W_{k,I}^\dagger\| \nonumber\\
& = \|\Sigma_k\Sigma_{k}^{-1}W_{k,I}^\dagger\| \nonumber\\
& = \|W_{k,I}^\dagger\|.
\end{align} 
The second equality follows from the unitary invariance of the norm in question; to see this, write $W_k = WP$, where $W$ is the $m\times m$ orthonormal basis from the full SVD of $A$, and $P = \begin{bmatrix}I_{k\times k}\\ 0\end{bmatrix}$; subsequently, the norm in question will be the norm of $\begin{bmatrix}W_{k,I}^\dagger\\ 0\end{bmatrix}$, which is $\|W_{k,I}^\dagger\|$.  A word of caution: Equation \eqref{EQARNorm} is not true if $W_{k,I}$ is replaced by $W_{A,I}$ the row submatrix of the full left singular vector matrix of $A$.

By a directly analogous calculation, we have that
\[
\|C^\dagger A\|= \|(V_{k,J}^*)^\dagger\|,
\]
whereupon the conclusion follows from the fact that $(V_{k,J}^*)^\dagger = (V_{k,J}^\dagger)^*$, which has the same norm as $V_{k,J}^\dagger$.
\end{proof}

Unfortunately, it is difficult to say much about the norms of pseudoinverses of submatrices of the truncated SVD of a matrix; however, we will give some indications later of some universal bounds that can be used in certain cases.

The following proposition gives a first estimate of the performance of the CUR approximation suggested by the exact CUR decomposition of Theorem \ref{THM:CUR} in terms of the underlying CUR decomposition of $A$.

\begin{proposition}\label{PROP:Perturbation}
Let $\tilde{A}=A+E$ for a fixed but arbitrary $E\in\K^{m\times n}$. Suppose that $A,C,U,R,\tilde{C},\tilde{U},$ and $\tilde{R}$ are given as in \eqref{EQ:tildes}.  For any submultiplicative norm $\|\cdot\|$ on $\K^{m\times n}$,
\begin{multline*} \|A-\tilde{C}\tilde{U}^{\dagger}\tilde{R}\|\leq \|A-CU^\dagger R\|+\|C\tilde{U}^\dagger\|\|E_I\|+\|\tilde{U}^\dagger \tilde{R}\|\|E_J\| \\ +\|CU^\dagger\|\|U^\dagger R\|\|E_{I,J}\|(3+\|\tilde{U}^\dagger\|\|E_{I,J}\|).\end{multline*}
\end{proposition}
\begin{proof}
Begin with the fact that \[ \|A-\tilde{C}\tilde{U}^\dagger \tilde{R}\|\leq \|A-CU^\dagger R\|+\|CU^\dagger R-\tilde{C}\tilde{U}^\dagger \tilde{R}\|.\]
Then we have
	\begin{eqnarray*}
	\|CU^\dagger R-\tilde{C}\tilde{U}^\dagger \tilde{R}\|&\leq&\|CU^\dagger R-C\tilde{U}^\dagger R\|+\|C\tilde{U}^\dagger R -C\tilde{U}^\dagger \tilde{R}\|+\|C\tilde{U}^\dagger \tilde{R}-\tilde{C}\tilde{U}^\dagger \tilde{R}\|\\
	&\leq&\|CU^\dagger R-C\tilde{U}^\dagger R\|+\|C\tilde{U}^\dagger\|\| R- \tilde{R}\|+\|C - \tilde{C}\|\|\tilde{U}^\dagger \tilde{R}\| \\
    & = & \|CU^\dagger R-C\tilde{U}^\dagger R\|+\|C\tilde{U}^\dagger\|\|E_I\|+\|\tilde{U}^\dagger \tilde{R}\|\|E_J\|.
	\end{eqnarray*}

To estimate the first term above, recall from Lemma \ref{LEM:Projections} that $C = CC^\dagger C = CU^\dagger U$, and likewise $R=RR^\dagger R = UU^\dagger R$.  Consequently, since $E_{I,J} = \tilde{U}-U$, the following holds:
\begin{eqnarray}\label{EQCURDiff}
	\|CU^\dagger R-C\tilde{U}^\dagger R\|&=&\|CU^\dagger UU^\dagger R-CU^\dagger(\tilde{U}-E_{I,J})\tilde{U}^\dagger(\tilde{U}-E_{I,J})U^\dagger R\| \nonumber\\
	&\leq&\|CU^\dagger(U-\tilde{U})U^\dagger R\|+\|CU^\dagger\tilde{U} \tilde{U}^\dagger E_{I,J} U^\dagger R\|+ \nonumber\\
	&&\|CU^\dagger E_{I,J}\tilde{U}^\dagger \tilde{U} U^\dagger R|+\|CU^\dagger E_{I,J}\tilde{U}^\dagger E_{I,J} U^\dagger R\|.
	\end{eqnarray}
The first term above is evidently at most $\|CU^\dagger\|\|U^\dagger R\|\|E_{I,J}\|$, whereas the second is majorized by the same quantity on account of the fact that $\tilde{U}\tilde{U}^\dagger$ is a projection.  Similarly, as $\tilde{U}^\dagger\tilde{U}$ is a projection, the third term in \eqref{EQCURDiff} is at most $\|CU^\dagger\|\|U^\dagger R\|\|E_{I,J}\|$, while the final term is at most $\|CU^\dagger\|\|U^\dagger R\|\|\tilde{U}^\dagger\|\|E_{I,J}\|^2.$  Putting these observations together, and combining \eqref{EQCURDiff} with Proposition \ref{PROP:NormTerms} yields the following:

 \begin{equation}\label{EQ:CUR-CtildeUR}
    \|CU^\dagger R-C\tilde{U}^\dagger R\| \leq \|CU^\dagger\|\|U^\dagger R\|(3\|E_{I,J}\|+\|\tilde{U}^\dagger\|\|E_{I,J}\|^2).
	\end{equation}
	Combining the estimates of \eqref{EQCURDiff} and \eqref{EQ:CUR-CtildeUR} yield the desired conclusion.
\end{proof}

Note that if columns and rows are chosen so that a valid CUR decomposition of $A$ is obtained (i.e. $A=CU^\dagger R$), then the corresponding norm term in the above proposition is 0.  Proposition \ref{PROP:Perturbation} gives only a preliminary estimate, but is also flexible since it allows the use of any submultiplicative norm.  It should be noted that while the decomposition considered here is the direct analogue of that in Theorem \ref{THM:CUR}, there is one key difference due to the presence of noise: namely that the rank of $\tilde{U}$ is typically larger than the rank of $A$ provided more than $\rank(A)$ columns or rows are chosen.  Therefore, $\tilde{C}\tilde{U}^\dagger\tilde{R}$ is an approximation of $A$ which has larger rank. It is natural to consider then what happens if the target rank is enforced.  By modifying the proof of Proposition \ref{PROP:Perturbation}, we arrive at the following.   Throughout the rest of this section, we will assume that $\rank(U)=k$ and hence $A=CU^\dagger R$; otherwise the same estimates hold with the additional term $\|A-CU^\dagger R\|$ appearing on the right-hand side.  In Section \ref{SEC:ColumnSelection} we will illustrate how columns and rows may be randomly sampled from $\tilde{A}$ to guarantee that this assumption is valid with high probability. 

\begin{proposition}\label{PROP:PB}
With the notation and assumptions of Proposition \ref{PROP:Perturbation}, suppose $\rank(A)=\rank(U)=k$, and let $\tilde{U}_k$ be the best rank-$k$ approximation of $\tilde{U}$.  Then for any submultiplicative norm $\|\cdot\|$ on $\K^{m\times n}$,
\[ \|A-\tilde{C}\tilde{U}_k^{\dagger}\tilde{R}\|\leq \|C\tilde{U}_k^\dagger\|\|E_J\|+\|\tilde{U}_k^\dagger \tilde R\|\|E_I\| +\|CU^\dagger\|\|U^\dagger R\|(3\|U-\tilde{U}_k\|+\|\tilde{U}_k^\dagger\|\|U-\tilde{U}_k\|^2).\]
\end{proposition}
\begin{proof}
The proof is the same as that of Proposition \ref{PROP:Perturbation} \textit{mutatis mudandis}; we note for the reader's convenience that the replacement for the use of $E_{I,J}=\tilde{U}-U$ in the final estimations yielding \eqref{EQ:CUR-CtildeUR} is the quantity $U-\tilde{U}_k$.
\end{proof}

The presence of terms depending on $\tilde{U}$ in the error bounds above are undesirable, so we now are tasked with estimating them.  Before stating the final bound, we estimate some of the terms specifically in the following lemma.

\begin{lemma}\label{LEM:PBTerms}
With the notations and assumptions of Proposition \ref{PROP:PB}, let $W_{k,I}$ and $V_{k,J}$ as in Proposition \ref{PROP:NormTerms}.  Suppose $\|\cdot\|$ satisfies the conditions of Theorem \ref{THM:Stewart}. Provided $\sigma_k(U)>2\mu\|E_{I,J}\|$, the following estimates hold:
\begin{enumerate}[(i)]
\item\label{ITEM:UkBound} $\|\tilde{U}_k^\dagger\|\leq \dfrac{\|U^\dagger\|}{1-2\mu\|U^\dagger\|_2\|E_{I,J}\|}$,

\medskip
\item\label{ITEM:CUkBound} 
$\|C\tilde{U}_k^\dagger\|\leq \dfrac{\|U^\dagger\|}{1-2\mu\|U^\dagger\|_2\|E_{I,J}\|}(2\|E_{I,J}\|\|W_{k,I}^\dagger\|)+\|W_{k,I}^\dagger\|$,

\medskip
\item\label{ITEM:UkRBound} 
$\|\tilde{U}_k^\dagger\tilde{R}\|\leq \dfrac{\|U^\dagger\|}{1-2\mu\|U^\dagger\|_2\|E_{I,J}\|}\left(2\|E_{I,J}\|\|V_{k,J}^\dagger\|+\|E_J\|\right)+\|V_{k,J}^\dagger\|.$

\end{enumerate}
\end{lemma}
\begin{proof}
To see item $(\ref{ITEM:UkBound})$, note that $\tilde{U}_k = U + (\tilde{U}_k-U)$, and notice that $\|U-\tilde{U}_k\|\leq\|U-\tilde{U}\|+\|\tilde{U}-\tilde{U}_k\|$, where the first term is equal to $\|E_{I,J}\|$ by definition, and the second satisfies $\|\tilde{U}-\tilde{U}_k\|\leq\|E_{I,J}\|$ by Mirsky's Theorem.  Hence $\|U-\tilde{U}_k\|\leq2\|E_{I,J}\|$.  Using this estimate, we see that if $\sigma_k(U)>2\mu\|E_{I,J}\|\geq\mu\|\tilde{U}_k-U\|$, then by Corollary \ref{COR:AdaggerBounds}, 
\[ \|\tilde{U}_k^\dagger\|\leq \dfrac{\|U^\dagger\|}{1-\mu\|U^\dagger\|_2\|\tilde{U}_k-U\|}\leq \dfrac{\|U^\dagger\|}{1-2\mu\|U^\dagger\|_2\|E_{I,J}\|},\] which is $(\ref{ITEM:UkBound})$.

 To see  $(\ref{ITEM:CUkBound})$, notice that $C=CU^\dagger U$ by Theorem \ref{THM:CUR}, whence applying Proposition \ref{PROP:NormTerms} yields

\begin{eqnarray*}
\|C\tilde{U}_k^\dagger\|&=&\|CU^\dagger U\tilde{U}_k^\dagger\| \\
&\leq&\|CU^\dagger\|\|U\tilde{U}_k^\dagger\|\\
&=&\|W_{k,I}^\dagger\|\|(U-\tilde{U}_k)\tilde{U}_k^\dagger+\tilde{U}_k\tilde{U}_k^\dagger\|\\
&\leq& \|W_{k,I}^\dagger\|(\|(U-\tilde{U}_k)\tilde{U}_k^\dagger\|+\|\tilde{U}_k\tilde{U}_k^\dagger\|)\\
&\leq&\|W_{k,I}^\dagger\|(\|U-\tilde{U}_k\|\|\tilde{U}_k^\dagger\|+1) \\
&\leq&\|W_{k,I}^\dagger\|(2\|E_{I,J}\|\|\tilde{U}_k^\dagger\|+1).
\end{eqnarray*}
Now \eqref{ITEM:CUkBound} follows from applying \eqref{ITEM:UkBound} to the above estimate.

Similarly, we have \[  \|\tilde{U}_k^\dagger R \|\leq \|U^\dagger R\|(2\|E_{I,J}\|\|\tilde{U}_k^\dagger\|+1).\]
Thus to prove $(\ref{ITEM:UkRBound})$, note that
\begin{eqnarray*}
\|\tilde{U}_k^\dagger \tilde{R} \|&\leq&\|\tilde{U}_k^\dagger R\|+\|\tilde{U}_k^\dagger E_J\| \\
&\leq&\|U^\dagger R\|(2\|E_{I,J}\|\|\tilde{U}_k^\dagger\|+1)+\|\tilde{U}_k^\dagger\|\|E_J\|,
\end{eqnarray*}
whereby applying Proposition \ref{PROP:NormTerms} and \eqref{ITEM:UkBound} yields the conclusion.
\end{proof}

\begin{theorem}\label{THM:PB}
Take the notations and assumptions of Lemma \ref{LEM:PBTerms}. If $\sigma_k(U)>2\mu\,\|E_{I,J}\|$,   

    \begin{multline*}\label{IEQTHm4.10}
        \|A-\tilde{C}\tilde{U}_k^{\dagger}\tilde{R}\|\leq \frac{\|U^\dagger\|}{1-2\mu\|U^\dagger\|_2\|E_{I,J}\|}\left\{2\|E_{I,J}\|\left[\|E_J\|\|W_{k,I}^\dagger\|+\|E_I\|\|V_{k,J}^\dagger\|+\right.\right.\\
        \left.\left.+2\|E_{I,J}\|\|W_{k,I}^\dagger\|\|V_{k,J}^\dagger\|\right]+\|E_I\|\|E_J\|\right\}\\
        + \|W_{k,I}^\dagger\|\|E_J\|+\|V_{k,J}^\dagger\|\|E_I\|+6\|W_{k,I}^\dagger\|\|V_{k,J}^\dagger\|\|E_{I,J}\|.
    \end{multline*}


\end{theorem}
\begin{proof}
Recalling that $\|U-\tilde{U}_k\|\leq2\|E_{I,J}\|$ as estimated in the proof of Lemma \ref{LEM:PBTerms}, the conclusion of the proof follows by combining this estimate with those of Propositions \ref{PROP:NormTerms}, \ref{PROP:PB}, and Lemma \ref{LEM:PBTerms}, and rearranging terms.  
\end{proof}

Note that all terms in curly braces in the bound of Theorem \ref{THM:PB} are second order in the noise, whereas the latter three terms are first order. 


\subsection{Refined Estimates}

One drawback of the estimates of Theorem \ref{THM:PB} is that the right-hand side maintains dependencies on the submatrix $U$ chosen.  Here we make a few remarks about certain cases in which more can be said.

First, $\|U^\dagger\|_2\geq\|A^\dagger\|_2$, which follows from singular value inequalities as in \cite[Theorem 1]{ThompsonInterlacing}; this inequality can be used in the denominator of the fractional term in Theorem \ref{THM:PB}.  

Second, if one assumes that maximal volume submatrices of the left and right singular values are chosen, then one can use estimates from \cite{Osinsky2018} to give bounds on the corresponding spectral norms.  Recall that the volume of a matrix $B\in\K^{m\times n}$ is $\prod\sigma_i(B)$.

\begin{proposition}\label{PROP:MaxVol}
Take the notations and assumptions of Proposition \ref{PROP:PB} and let $W_{k,I}$ and $V_{k,J}$ be the submatrices of $W_k$ and $V_k$ such that $W_{k,I}$ has maximal volume among all $|I|\times k $ submatrices of $W_k$ and $V_{k,J}$ is of maximal volume among all $|J|\times k$ submatrices of $V_k$. Then 

\begin{equation}\label{EQ:MaxVol}
\|W_{k,I}^\dagger\|_2 \leq \sqrt{1+\frac{k(m-|I|)}{|I|-k+1}},\quad\|V_{k,J}^\dagger\|_2\leq \sqrt{1+\frac{k(n-|J|)}{|J|-k+1}}.
\end{equation}
Moreover,
\begin{equation}\label{EQ:MaxVol2} \|U^\dagger\|_2\leq \sqrt{1+\frac{k(m-|I|)}{|I|-k+1}} \sqrt{1+\frac{k(n-|J|)}{|J|-k+1}}\|A^\dagger\|_2.\end{equation}
\end{proposition}

Note that \eqref{EQ:MaxVol} appears in \cite{Osinsky2018}, and the moreover statement follows by Proposition \ref{PROP:CAR} and the assumption that $\rank(U)=\rank(A)$.  For ease of notation, since the upper bounds appearing in \eqref{EQ:MaxVol} are universal, we abbreviate the quantities there $t(k,m,|I|)$, and $t(k,n,|J|)$, respectively as in \cite{Osinsky2018}.  Regard also, that Frobenius bounds are also provided in \cite{Osinsky2018}, where the upper bound is $\tilde{t}(k,m,|I|) = \sqrt{k+\frac{k(m-|I|)}{|I|-k+1}}$.

\begin{remark}
If columns and rows of $A$ are chosen to give the maximal volume submatrices $W_{k,I}$ and $V_{k,J}$ as prescribed in Proposition \ref{PROP:MaxVol}, then the error bounds in Theorem \ref{THM:PB} may be replaced with the corresponding quantities in \eqref{EQ:MaxVol} and \eqref{EQ:MaxVol2}, which are dependent primarily upon the rank and size of $A$.  Note that this requires the assumption that $\|E_{I,J}\|_2\leq 1/\left(2\mu\,t(k,m,|I|)t(k,n,|J|)\|A^\dagger\|_2 \right)$ as well.  Thus the upper bounds maintain dependence on the norm of $A^\dagger$, but not explicitly on the norm of the submatrix $U$.  
\end{remark} 

Additionally, some of our estimates above are somewhat crude, in that we always used the inequality $\|AB\|\leq\|A\|\|B\|$ for the submultiplicative norms; however, we could have used the bound $\|AB\|\leq\|A\|_2\|B\|$, which gives a better estimate since $\|A\|_2\leq\|A\|$ for the types of norms allowed by our results here.  

\begin{remark}
Since $U^\dagger = W_{k,I}^\dagger \Sigma_k^\dagger V_{k,J}^\dagger$ and $\|\Sigma_k^\dagger\|=\|A^\dagger\|$, we may replace the fractional term in Theorem \ref{THM:PB} with
\[ \dfrac{\|W_{k,I}^\dagger\|\|V_{k,J}^\dagger\|\|A^\dagger\|}{1-2\mu\|W_{k,I}^\dagger\|\|V_{k,J}^\dagger\|\|A^\dagger\|\|E_{I,J}\|}\]
thus giving a bound independent of the chosen $U$.  Indeed, this means that the error bounds in Theorem \ref{THM:PB} are of the form
\[ \|A-\tilde{C}\tilde{U}_k^\dagger\tilde{R}\|\leq\|A-CU^\dagger R\| + O(\|E\|) + O(\|A^\dagger\|\|E\|^2). \]
That is, the first order terms depend essentially only on the noise, whereas the second order terms have dependence on $\|A^\dagger\|$.  Do note that the assumptions in Theorem \ref{THM:PB} do imply that $\|E\|\|A^\dagger\|\leq C_1$ for some universal constant; on the other hand, it could be that this quantity is rather small, and so we leave the expression as is to denote the second order dependence on the noise matrix.
\end{remark}

In this section, we illustrated one way to enforce the rank of the CUR approximation, but it has recently been suggested by some authors that a better way to do so would be to consider $(CU^\dagger R)_k$, which means to make the CUR approximation suggested by Theorem \ref{THM:CUR}, and then take its best rank $k$ approximation \cite{TroppNystrom,BeckerNystrom}.  These works are for the Nystr\"{o}m method which is the special case of CUR when $A$ is symmetric positive semi-definite. It would be interesting in the general CUR case to determine if this method of enforcing the rank performs better or not; this task we leave to future work.

\section{Row and column selection for the  CUR Decomposition}\label{SEC:ColumnSelection}

One important question left unanswered by the discussion in the previous section is: given a matrix, how should one go about choosing columns and rows so that a good CUR approximation is obtained?  This has been the objective of a substantial portion of the CUR literature, and has brought forth several interesting results along the way.  As a preliminary note: consider a worst case when $A$ is a matrix with a single large nonzero entry.  Here, we must choose the column and row which contain this element or else the resultant CUR approximation will be 0, and hence far from the initial matrix.  Thus it is evident that in many scenarios na\"{i}vely sampling columns and rows uniformly could be arbitrarily bad, so it is beneficial to take into account some structural information of $A$.

\subsection{ Row and Column Selection for the exact CUR decomposition} Here we ask the question: given a low rank matrix $A$, how should we choose rows and columns to obtain a valid CUR decomposition as in Theorem \ref{THM:CUR}? In other words, how can we choose $C$ and $R$ so that the condition $\rank(A)=\rank(U)$ holds. Here we present one method of doing this; namely, we show that slightly oversampling rows and columns randomly is successful with high probability.  To state the theorem, we need the concept of the \textit{stable rank} \cite{tropp2009column}, also called \textit{numerical rank} \cite{Rudelson_2007}, of $A$, defined by \[\text{st.rank}(A):=\frac{\|A\|_F^2}{\|A\|_2^2}.\]  Note that this may be written as $\sum_{i=1}^k\frac{\sigma_i(A)^2}{\sigma_1(A)^2}$, whence evidently $\text{st.rank}(A)\leq\rank(A)$.   

One of the primary reasons for considering the stable rank of a matrix is that it is stable under small perturbations (whereas the rank is certainly not).  In particular, if $\tilde{A} = A+E$, then
\[ \text{st.rank}(A)\left(\dfrac{1-\|E\|_F/\|A\|_F}{1+\|E\|_2/\|A\|_F}\right)^2 \leq \text{st.rank}(\tilde{A})\leq \text{st.rank}(A)\left(\dfrac{1+\|E\|_F/\|A\|_F}{1-\|E\|_2/\|A\|_2}\right)^2.\]  The proof of this bound is a simple application of the triangle inequality and so is omitted. From these inequalities, we see that if $\|E\|_F/\|A\|_F$ and $\|E\|_2/\|A\|_2$ are small, then the stable ranks of $\tilde{A}$ and $A$ are close, implying the claim.

\begin{theorem}\label{THM:ColRowChnM}Suppose $A\in\mathbb{K}^{m\times n}$ with stable rank $r$. Let $\eps$ satisfy $0<\eps<\frac{\sigma_k(A)}{\|A\|_2}$, $\delta\in(0,1)$ and let $d_1\in[m]$, $d_2\in[n]$ satisfy
	\[
		d_1,d_2\gtrsim \left(\frac{r}{\eps^4\delta}\right)\log\left(\frac{r}{\eps^4\delta}\right).
		\]
		Let $R=A(I,:)$ be a $d_1\times n$ matrix consisting of $d_1$ rows of $A$ chosen independently with replacement, where row $i$ is chosen with probability $p_i = \frac{\|A(i,:)\|_2^2}{\|A\|_F^2}$.  Likewise, let $C=A(:,J)$ be a $m\times d_2$ matrix consisting of $d_2$ columns of $A$ chosen independently with replacement with probabilities $q_i = \frac{\|A(:,i)\|_2^2}{\|A\|_F^2}$, and let $U=A(I,J)$. Then with probability at least $(1-2\exp(-c/\delta))^2$ the following holds: 
		\[
		\rank(U)=k \text{ and } A=CU^\dagger R.
		\]
		Moreover, the conclusion of the theorem also holds if we take $I_0$ and $J_0$ to be the indices of $I$ and $J$ above without repeated entries.
	\end{theorem}

	Before giving the proof of Theorem \ref{THM:ColRowChnM}, we will state some simple lemmas beginning with one that is derived from the proof of \cite[Theorem 1.1]{Rudelson_2007}.
	
	\begin{proposition}[\cite{Rudelson_2007}]\label{THM_Rudelson}
		Let $A\in\K^{m\times n}$ have stable rank $r$.  Let $\eps,\delta\in(0,1)$, and let $d\in [m]$ satisfy
		\[
		d\gtrsim \left(\frac{r}{\eps^4\delta}\right)\log\left(\frac{r}{\eps^4\delta}\right).
		\]
		Consider a $d\times n$ matrix $\hat{R}$, which consists of $d$ normalized rows of $A$ picked independently with replacement, where row $i$ is chosen with probability $p_i = \frac{\|A(i,:)\|_2^2}{\|A\|_F^2}$. Then with probability at least $1-2\exp(-c/\delta)$, 
		\[
		\|A^*A-\hat{R}^*\hat{R}\|_2\leq\frac{\eps^2}{2}\|A\|_2^2.
		\]
	\end{proposition}

\begin{proof}[\textbf{Proof of Theorem \ref{THM:ColRowChnM}}]
The proof is essentially a corollary of Proposition \ref{THM_Rudelson}.  Note that by Corollary \ref{CR_RANK_U} it suffices to show that $\rank(C)=\rank(R)=k$.  To utilize Proposition \ref{THM_Rudelson}, let $\hat{C}$ and $\hat{R}$ be normalized versions of $C$ and $R$, respectively, and note that $\rank(\hat{R})=\rank(R)$ and $\rank(\hat{C})=\rank(C)$.  By Proposition \ref{THM_Rudelson} and the assumption on $\eps$, with probability at least $1-2\exp(-c/\delta)$ the following holds: 
		\begin{equation}\label{EQ:RudelsonVershynin}
		\|A^*A-\hat{R}^* \hat{R}\|_2\leq \frac{\eps^2}{2}\|A\|_2^2< \frac{1}{2}\sigma_{k}^2(A)<\sigma_k^2(A).
		\end{equation} 
Therefore $\rank(\hat{R})\geq k$.  In addition, $\rank(\hat{R})\leq\rank(A)=k$. Hence, $\rank(\hat{R})=k$.	

Using the same argument again, we can conclude that with probability at least $1-2\exp(-c/\delta)$,  $\rank(\hat{C})=k$.  Thus with probability at least $\left(1-2\exp(-c/\delta) \right)^2$, $\rank(U)=k$, and so $A=CU^\dagger R$.   The moreover statement follows from the fact that repeated columns and rows do not affect the validity of the statement $A=CU^\dagger R$ as mentioned in Remark \ref{REM:CUR}. 
\end{proof}

Let us stress the point that the choices of columns and rows in Theorem \ref{THM:ColRowChnM} are independent of each other, and hence the complexity of the algorithm implied by the theorem is low.  Note that the sampling complexity in Theorem \ref{THM:ColRowChnM} ostensibly depends on the stable rank of $A$, which is a reduction from most sampling methods for CUR approximations (see Section \ref{SEC:LiteratureSurvey} for a survey).  However, our assumption on $\eps$ implies that $\frac{r}{\eps^4}\geq k$, in which case our sampling complexity is at least
$\frac{k}{\delta}\log\left(\frac{k}{\delta}\right)$. 
Thus Theorem \ref{THM:ColRowChnM} implies that sampling $\frac{k}{\delta}\log(\frac{k}{\delta})$ rows and columns of $A$ yields a CUR decomposition.  Most results in the CUR approximation literature (e.g. \cite{DKMIII}) require sampling $\frac{k}{\eps^\alpha\delta}\log(\frac{k}{\eps^\alpha\delta})$ rows and columns for some $\alpha$ where the $\eps$ is the same as ours.  Thus our sampling bound could not be derived from the existing ones without being of higher order.

\subsection{Putting it All Together -- Sampling Stability}

The bounds given in Section \ref{SEC:Perturbation} assumed that from the noisy observation $\tilde{A}$, we achieved an exact CUR decomposition of the low-rank matrix $A$ which we have no knowledge of \textit{a priori}.  Whereas Theorem \ref{THM:ColRowChnM} provides a way of randomly sampling columns of $A$ to achieve an exact CUR decomposition with high probability, some notion of stability of sampling in the presence of noise is needed to achieve our goal in the noisy case.

To begin, we show how the proof of the main theorem in \cite{Rudelson_2007} can be modified to admit other sampling probabilities.  

\begin{theorem}\label{THM:Stability}
Let $A\in\K^{m\times n}$ be fixed and have stable rank $r$.  Suppose that $\tilde{p}$ is a probability distribution satisfying $\tilde{p_i}\geq\alpha_i^2\frac{\|A(i,:)\|_2^2}{\|A\|_F^2}$ for all $i\in[m]$ for some constants $\alpha_i>0$ (with the convention that $\alpha_i=1$ if $A(i,:)=0$).  Let $\alpha:=\min\alpha_i$, and let $\eps,\delta\in(0,1)$ be such that $\frac{\eps^2\sqrt{\delta}}{2}<\alpha$. Let $d\in[m]$ satisfy 
\[ d\gtrsim \left(\frac{r}{\eps^4\delta}\right)\log\left(\frac{r}{\eps^4\delta}\right), \]
and let $\hat{R}$ be a $d\times n$ matrix consisting of normalized rows of $A$ chosen independently with replacement according to $\tilde{p}$.  Then with probability at least $1-2\exp(-\frac{c\alpha^2}{\delta})$ (which is at least $1-2\exp(-c\eps^4)$),
\[\|A^*A-\hat{R}^*\hat{R}\|_2\leq\frac{\eps^2}{2}\|A\|_2^2. \]
\end{theorem}

The proof of Theorem \ref{THM:Stability} requires a simple modification of the proof of the main theorem in \cite{Rudelson_2007}.  For completeness, we give the proof in Appendix \ref{SEC:AppendixProofStability}.

This brings us to our main stability theorem about exact CUR decompositions.
\begin{theorem}\label{THM:StableCUR}
Let $A\in\K^{m\times n}$ be fixed and have stable rank $r$ and rank $k$.  Suppose that $\tilde{p},\tilde{q}$ are probability distributions satisfying $\tilde{p_i}\geq\alpha_i^2\frac{\|A(i,:)\|_2^2}{\|A\|_F^2}$ and $\tilde{q_j}\geq\beta_j^2\frac{\|A(:,j)\|_2^2}{\|A\|_F^2}$ for all $i\in[m]$ and $j\in[n]$ for some constants $\alpha_i,\beta_i>0$ (with the convention that $\alpha_i=1$ if $A(i,:)=0$ and $\beta_j=1$ if $A(:,j)=0$).  Let $\alpha:=\min\alpha_i$, $\beta:=\min\beta_i$, and let $\eps,\delta\in(0,1)$ be such that $\frac{\eps^2\sqrt{\delta}}{2}<\min\{\alpha,\beta\}$ and $\eps<\frac{\sigma_k(A)}{\|A\|_2}$. Let $d_1\in[m],d_2\in[n]$ satisfy 
\[ d_1,d_2\gtrsim \left(\frac{r}{\eps^4\delta}\right)\log\left(\frac{r}{\eps^4\delta}\right), \]
and let $R=A(I,:)$ be a $d_1\times n$ row submatrix of $A$ consisting of rows chosen independently with replacement according to $\tilde{p}$. Likewise, let $C=A(:,J)$ be a $m\times d_2$ column submatrix consisting of $A$ whose columns are chosen independently with replacement according to $\tilde{q}$, and let $U=A(I,J)$.  Then with probability at least $(1-2\exp(-\frac{c\alpha^2}{\delta}))(1-2\exp(-\frac{c\beta^2}{\delta}))$, the following hold:
\[ \rank(U) = k\;\;\text{and}\;\;A=CU^\dagger R.\]
	Moreover, the conclusion of the theorem also holds if we take $I_0$ and $J_0$ to be the indices of $I$ and $J$ above without repeated entries.
\end{theorem}
\begin{proof}
The proof is the same as that of Theorem \ref{THM:ColRowChnM} \textit{mutatis mudandis}, where one applies Theorem \ref{THM:Stability} rather than Theorem \ref{THM_Rudelson} to conclude that $\rank(C)=\rank(R)=k$.
\end{proof}

Note that by assumption, the probability of success in Theorem \ref{THM:StableCUR} is at least $(1-2\exp(-c\eps^4))^2$.  Additionally, the hypotheses of Theorem \ref{THM:StableCUR} imply that uniform sampling of rows and columns yields a CUR decomposition of $A$ with high probability.  Indeed, more generally, Theorem \ref{THM:StableCUR} implies that as long as $\tilde{p}_i,\tilde{q_i}\geq\eps_0>0$, then sampling columns and rows according to these probabilities yields a valid CUR decomposition with high probability as long as $\eps^2\sqrt{\delta}<2\eps_0$.

Now we may use Theorem \ref{THM:StableCUR} to provide guarantees for when an underlying CUR decomposition is obtained for a CUR approximation of a low-rank plus noise matrix.  Our first observation is the following, which essentially states that uniformly sampling rows and columns of $\tilde{A}$ still yields $A=CU^\dagger R$ with high probability; this is the first result of this kind that does not use any additional assumptions about the matrix $\tilde{A}$ such as coherency (e.g. \cite{DemanetWu}).

\begin{corollary}\label{COR:Uniform}
Let $\tilde{A}=A+E$ with $A$ having stable rank $r$ and rank $k$.  Let $\alpha:= \frac{1}{\sqrt{m}}\min\{\frac{\|A\|_F}{\|A(i,:)\|_2}:A(i,:)\neq0\}$ and $\beta:=\frac{1}{\sqrt{n}}\min\{\frac{\|A\|_F}{\|A(:,j)\|_2}:A(:,j)\neq0\}$.  Let $\eps,\delta\in(0,1)$ be such that $\frac{\eps^2\sqrt{\delta}}{2}<\min\{\alpha,\beta\}$ and $\eps<\frac{\sigma_k(A)}{\|A\|_2}$.  Then sampling $d_1,d_2\geq C\left(\frac{r}{\eps^4\delta}\right)\log\left(\frac{r}{\eps^4\delta}\right)$ columns and rows of $\tilde{A}$ uniformly with replacement yields $\tilde{C},\tilde{U},\tilde{R}$ such that with probability at least $(1-2\exp(-\frac{c\alpha^2}{\delta}))(1-2\exp(-\frac{c\beta^2}{\delta}))$, which is at least $(1-2\exp(-c\eps^4))^2$,
\[\rank(U) = k\;\;\text{and}\;\;A=CU^\dagger R.\]
\end{corollary}
\begin{proof}
With the definitions of $\alpha,\beta$, we have that $\tilde{p_i}=\frac1m \geq \alpha_i^2\frac{\|A(i,:)\|_2^2}{\|A\|_F^2}$ where $\alpha_i^2 = \frac1m \frac{\|A\|_F^2}{\|A(i,:)\|_2^2}$.  The analogous statement holds for $\tilde{q}_i=\frac1n$, whereby an appeal to Theorem \ref{THM:StableCUR} yields the desired conclusion.
\end{proof}

\begin{remark}\label{REM:Atilde}
Without writing the full statement, let us note that another corollary is that by the same estimate as we did for the stable rank of $\tilde{A}$, we have that 
\[ \tilde{p}_i = \frac{\|\tilde{A}(i,:)\|_2^2}{\|\tilde{A}\|_F^2} \geq \left(\frac{1-\frac{\|E\|_F}{\|A\|_F}}{1+\frac{\|E(i,:)\|_2}{\|A(i,:)\|_2}}\right)^2\frac{\|A(i,:)\|_2^2}{\|A\|_F^2}=:\alpha_i^2p_i.\]
Thus sampling columns and rows of $\tilde{A}$ can ensure that the underlying CUR decomposition is valid for $A$ as long as $\eps,\delta$ are small enough. To achieve this, for example, one could try to estimate the signal to noise ratio to obtain an estimate for $\alpha_i$ in the above expression.
\end{remark}

\begin{remark}
Embedded in the assumptions of Corollary \ref{COR:Uniform} is a requirement about the sparsity of rows and columns of $A$, which one would expect to have in order to guarantee success of uniform sampling.  Indeed, consider the extreme case when $A$ consists of a single nonzero entry.  In this case, $\alpha=\frac{1}{\sqrt{m}}$, and the requirement on $\eps$ is such that approximately $m$ rows need to be sampled to guarantee that the single meaningful column is selected, which makes sense given the fact that the rows are chosen uniformly at random.
\end{remark}

\subsection{Survey of Sampling Methods}\label{SEC:LiteratureSurvey}
As mentioned previously, deterministically choosing columns and rows of a given matrix to form a good CUR approximation is often costly, and random sampling can give good approximations with much lower cost.  In this section, we will survey the results in the literature surrounding sampling of columns and rows; these algorithms are useful even when $A$ is not low rank. Here, we will not assume anything on the rank of $A$ or its stable rank, and so the letters $k,r,$ and $c$ will be used for the order of low rank approximation, number of chosen rows, and number of chosen columns, respectively.

\subsubsection{Randomized Algorithms}

For randomized sampling algorithms, there are two primary ways to sample: independently with replacement, or using Bernoulli trials.  There are essentially three main distributions used for sampling with replacement: uniform \cite{DemanetWu}, column and row lengths \cite{DKMIII,KannanVempala}, and leverage scores \cite{DMM08}.  The probability distributions on the columns are thus given by
\[p_i^\text{lev,k}:= \frac{1}{k}\|V_k(i,:)\|_2^2,\quad p_i^\text{col}:=\frac{\|A(:,i)\|_2^2}{\|A\|_F^2},\quad p_i^\text{unif}:=\frac{1}{n}, \quad i\in[n],\]
respectively; the distributions for the rows are defined analogously.  Note that for leverage scores, $A$ does not have to be rank $k$ in general, but the parameter $k$ determines how much the right singular vectors are truncated.

Evidently, uniform sampling is the easiest to implement, but it can fail to provide good results, especially if the input matrix is very sparse, for example.  On the other hand, leverage scores typically achieve the best performance because they capture the eigenspace structures of the data matrix, but this comes at the cost of a higher computational load to compute the distribution as it requires computing the truncated SVD of the initial data.  Column/row length sampling typically lies between both of the others in terms of performance as well as computational complexity.

In the algorithms which sample in this manner, the number of rows and columns chosen is fixed and deterministic, but when Bernoulli trials are used of course one only knows the expected number.

There are many randomized algorithms in the literature which attempt to construct a good CUR approximation.  Among these, there are norm guarantees for Frobenius error, and less commonly spectral error; there are relative and additive error guarantees, and there are many choices for the middle matrix which we have called $U$ beyond simply choosing $A(I,J)$ or $C^\dagger AR^\dagger$.  Since the literature is very scattered, we provide some summary here of the types of results in existence.

\subsection{Relative Error Bounds}

Relative error bounds are those of the form
\[\|A-CUR\|^2\leq f(\eps)\|A-A_k\|^2, \]
where $f$ is some hopefully small function of $\eps$, typically $1+\eps$.  Table \ref{TAB:RelativeError} provides a summary of the somewhat sparse literature giving relative error bounds, almost all of which are in the Frobenius norm.  Details about the choice of $U$ will be discussed following the table.

\begin{table}[h]
\centering
\scriptsize
		\caption{Relative Error CUR Approximations}
		\label{TAB:RelativeError}
		\begin{tabular}{||c|c|c|c|c|c|c|c|}
					\hline
				 Norm & $U$ & $f(\eps)$ & Sampling  & \# Columns (c) &\# Rows (r) & Complexity & Reference \\
					\hline \hline
     F &$(DA(I,J))^\dagger$ & $1+\eps$ & Leverage Scores & $O(\frac{k^2}{\eps^2})$ & $O(\frac{c^2}{\eps^2})$ & $O(mnk)$ & \cite{DMM08} \\
				\hline
				F &$(DA(I,J))^\dagger$ & $1+\eps$ & Leverage (Bernoulli) & $O(\frac{k\log k}{\eps^2})$ & $O(\frac{c\log c}{\eps^2})$ & $O(mnk)$ & \cite{DMM08} \\
				\hline
		 F &$C^\dagger AR^\dagger$ & $2+\eps$ & Leverage Scores & $O(\frac{k\log k}{\eps^2})$ & $O(\frac{k\log k}{\eps^2})$ & $O(mnk)$ &\cite{DMPNAS}  \\		
        \hline
        F &$MC^\dagger AR^\dagger$ & $1+\eps$ & Col/Row Lengths & $O(k+\frac{k}{\eps})$& $O(k+\frac{k}{\eps})$ & $O(\frac{mn^3k}{\eps})$ & \cite{BoutsidisOptimalCUR}  \\
				\hline
         F &$C^\dagger AR^\dagger$ & $1+\eps$ & Adaptive Sampling & $\frac{2k}{\eps}(1+o(1))$& $c+\frac{c}{\eps}$ & See below &\cite{WZ_2013}\tablefootnote{The error bounds in \cite{WZ_2013} are in expectation.}  \\
         \hline
         2 & $C^\dagger AR^\dagger$ & const. & DEIM & $k$ & $k$ & $O(mnk)$ & \cite{SorensenDEIMCUR}  \\
				\hline
		\end{tabular}
\end{table}

To the authors' knowledge, the results in Table \ref{TAB:RelativeError} are all of the ones available at present as relative error bounds are much more difficult to come by.  In \cite{DMM08}, the matrix $D$ is a diagonal scaling matrix which takes $A(I,J)$ and scales its $i$-th row by a scalar multiple of $1/p_i^{lev,k}$.  This scaling is done so that the probabilistic argument works, but the algorithm given therein cannot achieve an exact CUR decomposition even in the low rank case.  For \cite{DMM08}, Leverage Score sampling corresponds to sampling independently with replacement as described above, whereas Leverage (Bernoulli) means that Bernoulli trials utilizing the leverage scores are used to select an expected number of columns and rows.

Boutsidis and Woodruff \cite{BoutsidisOptimalCUR} give algorithms for computing optimal CUR approximations in several senses: they achieve optimal sampling complexity of $O(\frac{k}{\eps})$ and run in relatively low polynomial time, while providing relative error bounds in the Frobenius norm.  Some of their algorithms are randomized, but they also give a deterministic polynomial time algorithm for computing CUR approximations.  Essentially all of their approximations use $U = MC^\dagger AR^\dagger$, where $M$ is a judiciously (and laboriously) chosen matrix which enforces the desired rank, i.e. $\rank(U)=k$ where $k$ is given \textit{a priori}.  There are three algorithms given in \cite{BoutsidisOptimalCUR}, each of which has the same sampling complexity and error guarantees, so we only report one entry in Table \ref{TAB:RelativeError} (the running complexity reported is for the deterministic algorithm, but the randomized algorithms therein have smaller complexity).  One final note: the algorithms in \cite{BoutsidisOptimalCUR} are shown to give good CUR approximations with constant probability, but with very low constants (0.16 in one case), so to obtain a good approximation the algorithm should be run many times.

The complexity of the algorithm in \cite{WZ_2013} is 
$O((m+n)k^3\eps^{-\frac23}+mk^2\eps^{-2}+nk^2\eps^{-4}+mnk\eps^{-1})$.  The adaptive sampling procedure is a more sophisticated one which first uses the near-optimal column selection algorithm of Boutsidis, Drineas, and Magdon-Ismail \cite{BoutsidisNearOptimalCSSP} to select $O(\frac{k}{\eps})$ columns of $A$, then uses the same algorithm to select $c$ rows of $A$, say $R_1$, and the final step chooses $\frac{c}{\eps}$ more rows by sampling using leverage scores of the residual $A-AR_1^\dagger R_1$.  

 The DEIM method for selecting rows and columns chooses exactly $k$ columns and rows, with the tradeoff of only a constant error bound, which is given by $(\|W_{k,I}^{-1}\|_2+\|V_{k,J}^{-1}\|_2)^2$, where the index sets $I$ and $J$ are sets of size $k$ chosen via the DEIM algorithm \cite{SorensenDEIMCUR}, and $W_{k,I}$, $V_{k,J}$ are as in Proposition \ref{PROP:NormTerms}.  To our knowledge, this is the only relative error bound in the spectral norm.

An interesting paper by Yang et.~al.~gives sampling-dependent error bounds for the Column Subset Selection Problem, which can be used twice (once on $A$ and once on $A^*$) to give a CUR approximation \cite{Yang}.

\subsection{Additive Error Bounds}
Additive error bounds are those of the form
\[ \|A-CUR\|^2\leq \|A-A_k\|^2+g(\eps,A),\]
where $g$ is typically $O(\eps\|A\|^2)$.  Such guarantees are relatively easier to bome by compared to relative bounds; however, they are less useful in practice since $\eps\|A\|$ may be quite large. On the other hand, there exist spectral norm guarantees of this form as opposed to the case of relative error bounds which have only been found for the Frobenius norm.  Table \ref{TAB:AdditiveError} summarizes some of the canonical additive error bounds in existence.

\begin{table}[h]
\centering
\scriptsize
		\caption{Additive Error CUR Approximations choosing $c$ columns and $r$ rows.}
		\label{TAB:AdditiveError}
		\begin{tabular}{|c|c|c|c|c|c|c|c|}
					\hline
				Norm & $U$ & $g(\eps,A)$ &  Sampling  & $c$ & $r$ & Complexity & Ref \\
					\hline \hline
                    			$2$	& $(C^TC)_k^{-1} (C(I,:))^T $ &  $\eps \|A\|_F^2$ &               			Col/Row Lengths 
                    			 & $O(\frac{1}{\eps^4})$ & $ O(\frac{k}{\eps^2})$ & $O((m+r)c^2+nr+c^3)$ &\cite{DKMIII}\\ 
                                \hline
			$F$	& $(C^TC)_k^{-1} (C(I,:))^T $ & $\eps \|A\|_F^2$ & Col/Row Lengths & $O(\frac{k}{\eps^4})$ & $O(\frac{k}{\eps^2})$ & $O((m+r)c^2+nr+c^3)$ & \cite{DKMIII}\\
            \hline
			$2$ & $(A(I,J)_k)^\dagger$ &   $\eps \sum A_{ii}^2$ & $p_i = A_{ii}^2/\sum A_{jj}^2$ & $\frac{1}{\eps^4}$ & $\frac{1}{\eps^4}$& $O(m+c^3)$& \cite{DM05}$^{\ref{foot}}$		  \\	\hline	
			$F$ & $(A(I,J)_k)^\dagger$ &   $\eps \sum A_{ii}^2$ & $p_i = A_{ii}^2/\sum A_{jj}^2$ & $\frac{k}{\eps^4}$ & $\frac{k}{\eps^4}$& $O(m+c^3)$& \cite{DM05}\tablefootnote{\label{foot}In \cite{DM05}, the data matrix is required to be symmetric, positive semi-definite.}\\		\hline
			$2$	& $(A(I,J)_k)^\dagger $ & $\frac{\sqrt{mn}}{\mu k\log m}$ & Uniform & $O(\mu k\log m)$ & $O(\mu k \log m)$ & $\tilde{O}(k^3)$ & \cite{DemanetWu}\\
            \hline
		\end{tabular}
\end{table}

The restriction in \cite{DM05} to symmetric positive semi-definite matrices is common in the machine learning literature, as kernel and graph Laplacian matrices are of high importance in data analysis methods (e.g. spectral clustering, for one).  In this setting, the CUR approximations are of the form $CU^\dagger C^T$, and the method is called the Nystr\"{o}m method (see \cite{GittensMahoney} for an exposition and history).  Note also that the additive error $g(\eps,A)$ of \cite{DKMIII} is in fact the same (Frobenius norm) regardless of the norm the error is measured in.

The error bounds of Chiu and Demanet \cite{DemanetWu} reported in Table \ref{TAB:AdditiveError} require the additional assumption that the matrix $A$ has left singular vectors which are $\mu$--coherent, meaning that $\max_{i,j}|W_{i,j}|\leq \sqrt{\mu/n}$.  Coherency is a common assumption in Compressed Sensing (e.g. \cite{CandesRomberg}) and is a notion that the columns of an orthogonal basis are somewhat well spread out.  Additional bounds are given in \cite{DemanetWu}, some of which are more general than the one reported here, and some of which are tighter bounds under stricter assumptions, but for brevity we report only the main one in Table \ref{TAB:AdditiveError}.  Here, $\tilde{O}$ suppresses any logarithmic factors.

\section{Numerical simulations}\label{SEC:Numerics}

Here, we illustrate the performance of some of the CUR approximations mentioned previously on matrices of the form $\widetilde{A}=A+E$, where $A$ is low-rank and $E$ is a small perturbation matrix.  In the experiments, we will take $E$ to consist of i.i.d. Gaussian entries with mean zero and a given variance $\sigma^2$ which will change from experiment to experiment.  As a side note, we call the reader's attention to the fact that there is a package called rCUR for implementing various CUR approximations in R \cite{RCUR}.  The experiments here were performed in Matlab since the full flexibility of the rCUR package was not needed.

\begin{experiment}
First, we examine the effect of sampling scheme for the columns and rows and its effect on the approximation suggested by Theorem \ref{THM:CUR}.  We generate a Gaussian random matrix $A$ of size $500\times 500$  and force $A$ to be rank $k$, where $k$ varies from $1$ to $50$.  $A$ is perturbed by a Gaussian random matrix $E$ whose entries have standard deviation $10^{-5}$.  To make results easier to interpret, we normalize the matrices so that $\|A\|_2=1$, and $\|E\|_2=\sigma$ in each experiment. Based on the sampling results in Table \ref{TAB:RelativeError}, we choose $k\log k$ rows and columns of $A$ and form the corresponding matrices $\tilde{C},\tilde{U}$, and $\tilde{R}$.  We then compute the relative error $\|A- \tilde{C}\tilde{U}^\dagger \tilde{R}\|_2/\|A\|_2$.  Figure \ref{fig:MRankError_1000} shows the results, from which we see that when we sample $k\log k$ rows and columns, the relative error is essentially independent of the size of $k$ and of the sampling probabilities.

\begin{figure}[ht]
    \centering
    \begin{subfigure}[b]{0.49\linewidth}
		\includegraphics[width=\textwidth]{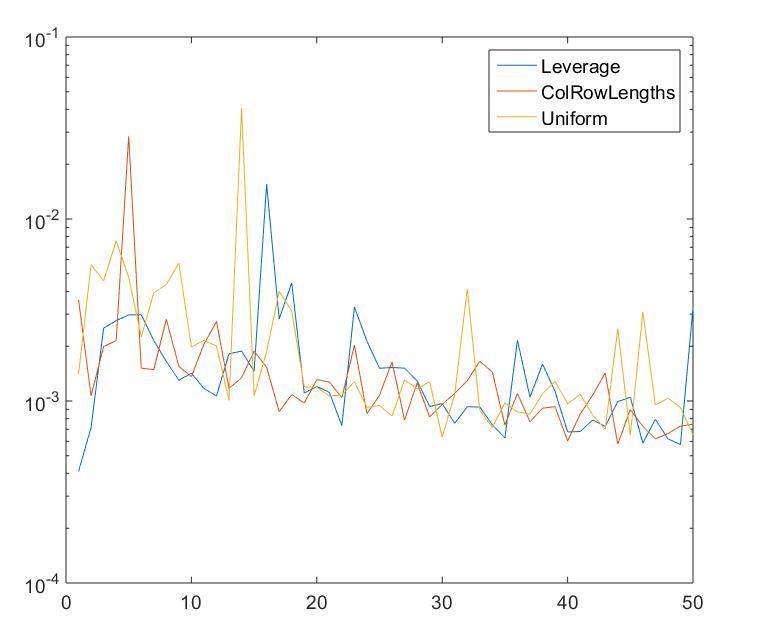}
		\caption{}	
	\label{FIG:RNN1}
	\end{subfigure}
	\begin{subfigure}[b]{0.49\linewidth}
		\includegraphics[width=\textwidth]{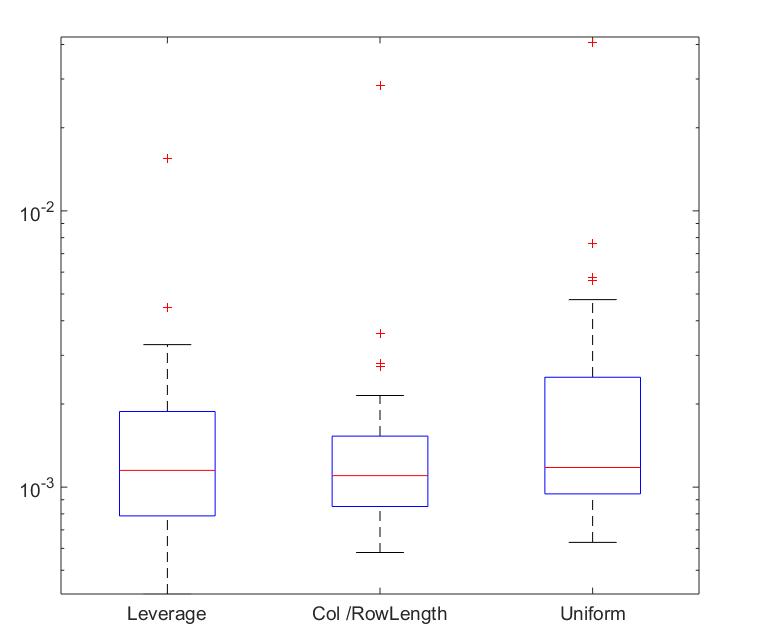}
		\caption{}
		\label{FIG:RNN2}
	\end{subfigure}
    \caption{(a)  $\|A-\tilde{C}\tilde{U}^\dagger\tilde{R}\|_2/\|A\|_2$ vs. $\rank(A)$; (b) Box plot of $\|A-\tilde{C}\tilde{U}^\dagger\tilde{R}\|_2/\|A\|_2$ with respect to different column/row sampling schemes.  The average is taken over all values of the number of columns and rows sampled.
    }
    \label{fig:MRankError_1000}
\end{figure} 
\end{experiment}

\begin{experiment}
To test how the sizes of the sampled submatrices influence the relative error, we consider a matrix of size $400\times 400$ with fixed rank $k=10$. We then choose $n$ columns and rows, where $n$ is a paramneter that varies from $15$ to $50$.  As in Experiment 1, we sample columns and rows according to the three probability methods described above. For each fixed $n$, we repeat the column and row sampling process $100$ times and consider the average of the relative errors.  Beginning with a random $A$ as in Experiment 1, Figure \ref{FIG:sizeofsubm_error} shows the results.   
In Figure \ref{FIG:sizeofsubm_error} (b), we see that the number of sampled columns and rows has relatively little effect on the error (as long as more than the target rank are chosen of course), whereas for the full matrix $A$, we see almost no difference in sampling schemes, with perhaps Column/Row length being slightly preferred. Figure \ref{FIG:sizeofsubm_error} (c) shows the same experiment for a sparse matrix $A$ and the results is as expected that uniform sampling yielded larger variation and error, which makes sense because unlike the other sampling schemes it is not unlikely that a $0$ column or row will be chosen. 

 \begin{figure}[h]
    \centering
     \begin{subfigure}[b]{0.32\linewidth}
		\includegraphics[width=\textwidth]{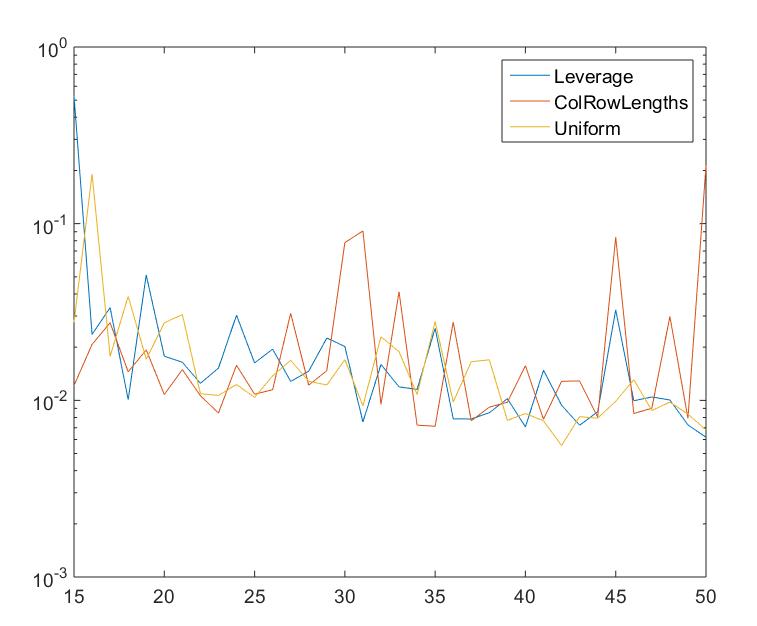}
		\caption{}	\label{FIG:sizeofsubm_error1}
	\end{subfigure}
	\begin{subfigure}[b]{0.32\linewidth}
		\includegraphics[width=\textwidth]{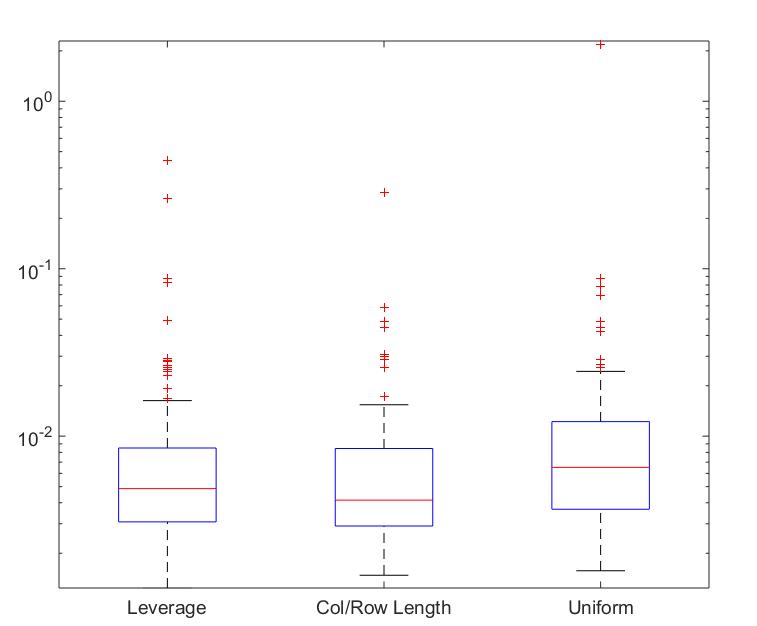}
		\caption{}
		\label{FIG:sizeofsubm_error2}
	\end{subfigure}
	\begin{subfigure}[b]{0.32\linewidth}
		\includegraphics[width=\textwidth]{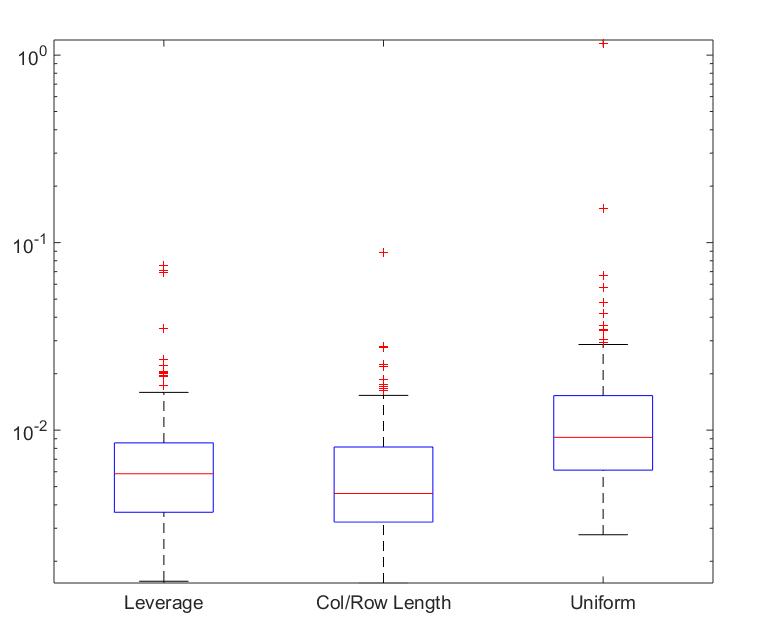}
		\caption{}
		\label{FIG:sizeofsubm_error_sp2}
	\end{subfigure}
    \caption{(a) $\|A-\tilde{C}\tilde{U}^\dagger\tilde{R}\|_2/\|A\|_2$ vs. $n$, the number of columns/rows chosen.  (b) Box plot for 100 trials of column/row sampling of relative error with respect to different column/row sampling schemes for $n=30$. (c) Same as (b) for a sparse random matrix $A$.}\label{FIG:sizeofsubm_error} 
\end{figure}

Figure \ref{fig:sizeofsubm_error_rd} shows the same experiment on a real data matrix coming from the Hopkins155 motion dataset \cite{Hopkins}. Here, we see the relative error decreasing and leveling out.  The given data matrix from Hopkins155 is approximately (but not exactly) rank 8 because it comes from motion data \cite{CosteiraKanade}.  The Hopkins155 data matrices are not normalized in any way, and this and other unreported experiments show that in that case, Leverage Score sampling tends to perform better.

\begin{figure}
    \centering
     \begin{subfigure}[b]{0.49\linewidth}
		\includegraphics[width=\textwidth]
		{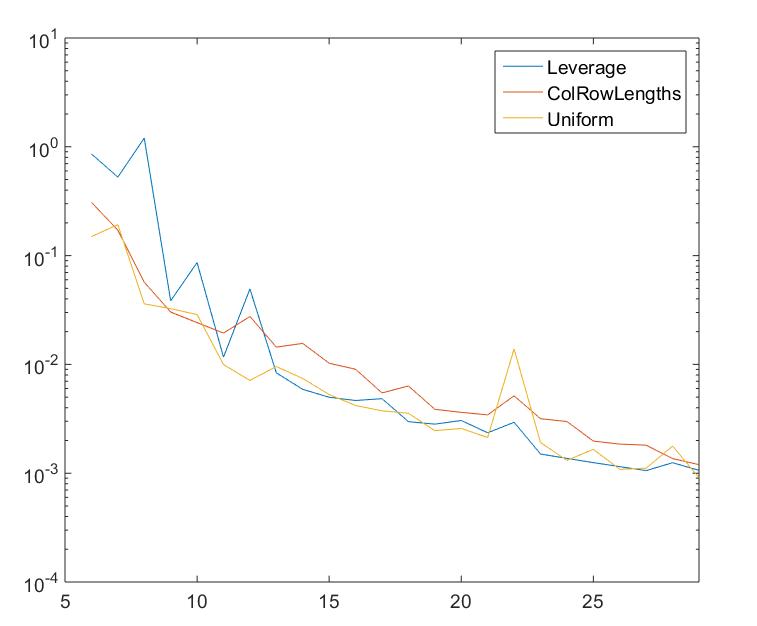}
		\caption{}	\label{FIG:sizeofsubm_error_rd1}
	\end{subfigure}
	\begin{subfigure}[b]{0.49\linewidth}
		\includegraphics[width=\textwidth]{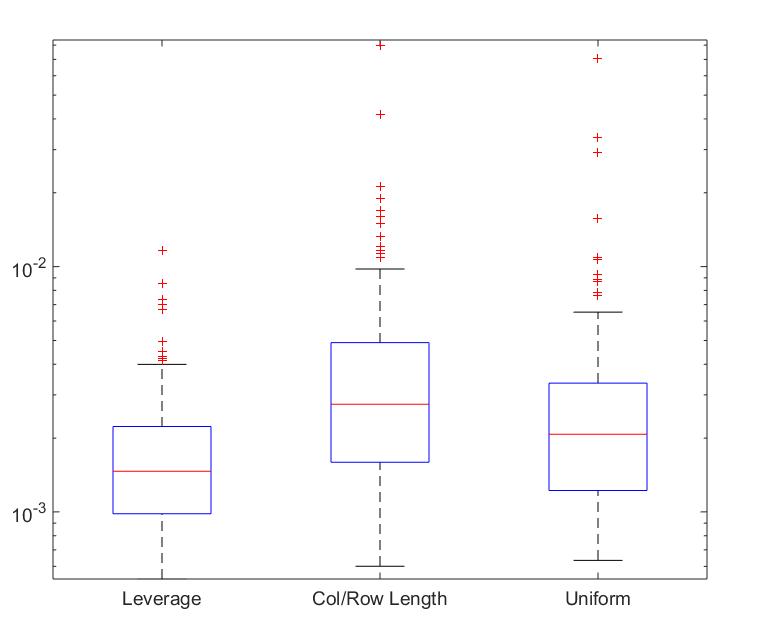}
		\caption{}
		\label{FIG:sizeofsubm_error_rd2}
	\end{subfigure}
    \caption{(a) Relative spectral error vs. number of columns and rows sampled for real data matrix from Hopkins155. (b) Box plot of relative spectral error with respect to different column/row sampling schemes for $n=21$.}
    \label{fig:sizeofsubm_error_rd}
\end{figure}
\end{experiment}

\begin{experiment}
In this simulation, we test how enforcing the low-rank constraint on $\tilde{U}$ will influence the relative error $\|A-\tilde{C}\tilde{U}_r^\dagger \tilde{R}\|_2/\|A\|_2$, where $r$ varies from $\rank(A)$ to $\rank(\tilde{U})$. We consider a $500\times 500$ matrix of rank $10$ perturbed by Gaussian noise with standard deviation $10^{-4}$. We randomly choose 60 columns and rows, and for each fixed $r$, we repeat this process 100 times and compute the average error. Figure \ref{fig:errorRankU} shows that if $r$ is closer to $\rank(A)$, the relative error is smaller as one might expect, while as the rank increases the error is saturated by the noise.

\begin{figure}
    \centering
    \begin{subfigure}[b]{0.49\linewidth}
		\includegraphics[width=\textwidth]{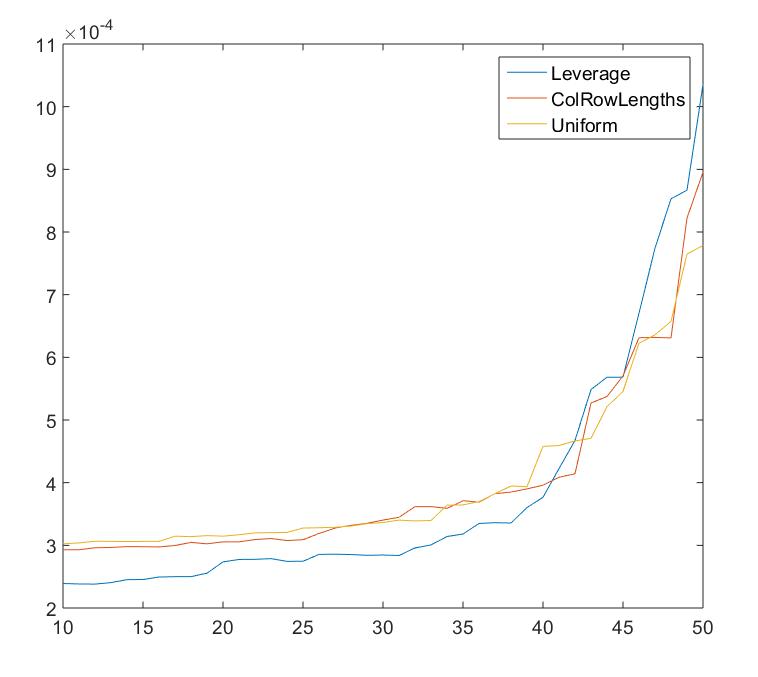}
		\caption{}	\label{FIG:errorRankU1}
	\end{subfigure}
	\begin{subfigure}[b]{0.49\linewidth}
		\includegraphics[width=\textwidth]
		{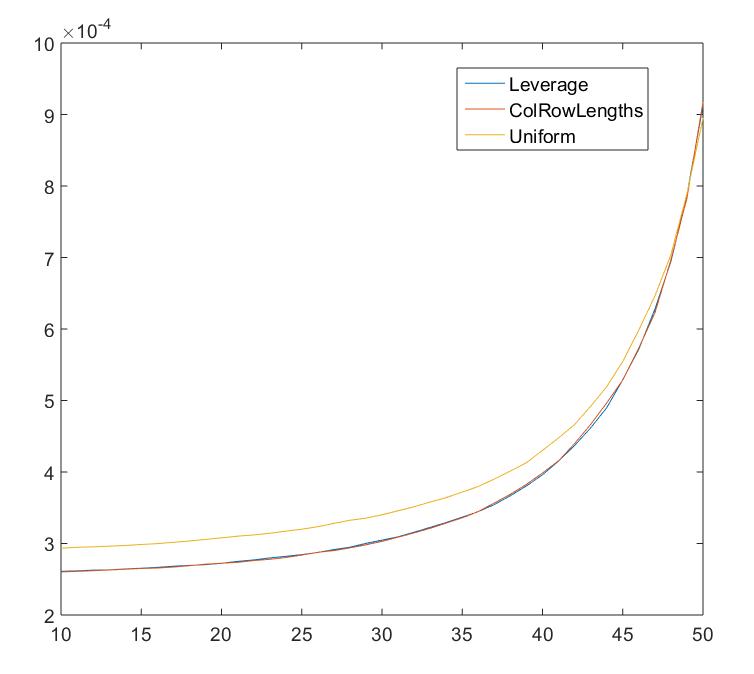}
		\caption{}
		\label{FIG:errorRankU2}
	\end{subfigure}
    \caption{(a)  $\|A-\tilde C\tilde{U}_r^\dagger\tilde{R}\|_2/\|A\|_2$ vs. $r$ which varies from $\rank(A)$ to $50$ for one specific choice of columns and rows, (b) Averaged errors over 100 trials of sampling columns/rows. }
    \label{fig:errorRankU}
\end{figure}
    
\end{experiment}

\begin{experiment}\label{EXP:ErrorBounds}
We test our analysis error bound in Proposition \ref{PROP:Perturbation}  and Theorem \ref{THM:PB} for $\|A-\tilde C \tilde U^\dagger \tilde R\|$. For this simulation, we consider a $500\times500$ matrix with rank 10. This matrix is perturbed  by a Gaussian random matrix with mean 0 and standard deviation $10^{-4}$. We choose 60 columns and rows via leverage scores, column/row lengths, and uniformly as before, and we compute the ratios between the right-hand side with the left-hand side in Proposition \ref{PROP:Perturbation}, Proposition \ref{PROP:PB}, and Theorem \ref{THM:PB}. This is repeated 200 times, and the average ratios are shown in Figure \ref{fig:CUR_TA_errorPro4.6} for Proposition \ref{PROP:Perturbation}, Figure \ref{fig:CUR_TA_errorPro4.7} for Proposition \ref{PROP:PB}, and Figure \ref{fig:CUR_TA_errorTHM4.9} for Theorem \ref{THM:PB}. 

We note that the error bounds for Propositions \ref{PROP:Perturbation} and \ref{PROP:PB} are relatively good, with the latter being slightly better due to enforcement of the rank.  Since many overestimates were made in Theorem \ref{THM:PB}, the ratio is somewhat higher; however, the estimates therein are general, and not overly pessimistic.


\begin{figure}[h]
    \centering
    \begin{subfigure}[b]{0.32\linewidth}
		\includegraphics[width=\textwidth]
		{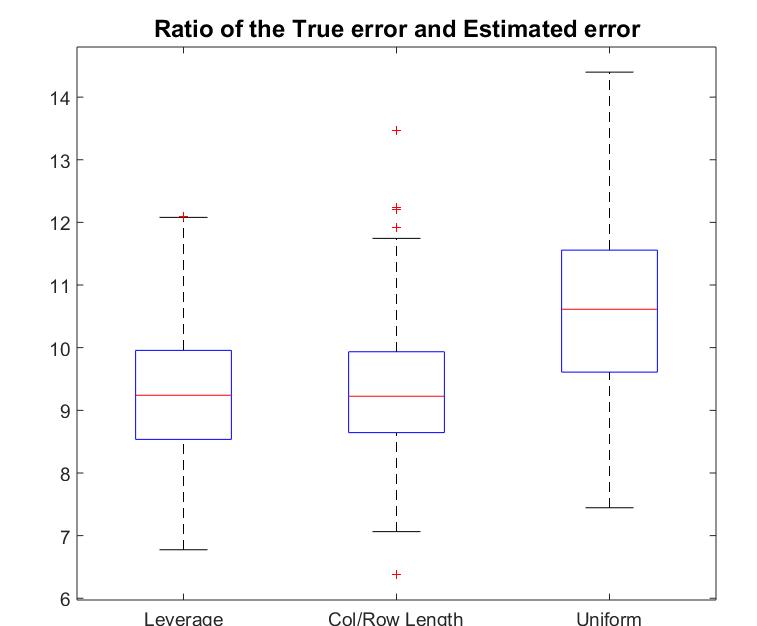}
		\caption{}	\label{fig:CUR_TA_errorPro4.6}
	\end{subfigure}
	\begin{subfigure}[b]{0.32\linewidth}
		\includegraphics[width=\textwidth]{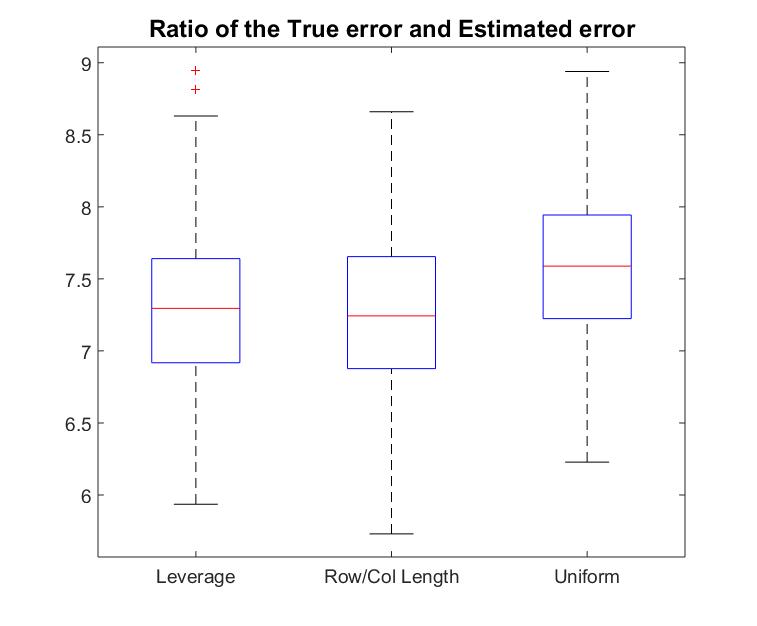}
		\caption{}
		\label{fig:CUR_TA_errorPro4.7}
	\end{subfigure}
	\begin{subfigure}[b]{0.32\linewidth}
		\includegraphics[width=\textwidth]{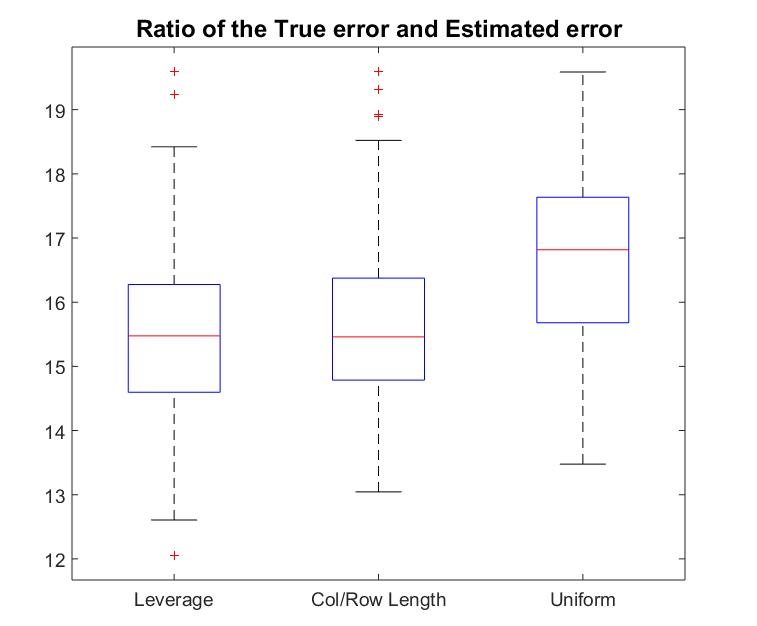}
		\caption{}
		\label{fig:CUR_TA_errorTHM4.9}
	\end{subfigure}
    \caption{(a) The ratio between our analytic error bounds  in Proposition \ref{PROP:Perturbation} and $\|A-\tilde{C}\tilde{U}^\dagger\tilde{R}\|$, (b) The ratio between our analytic error bounds  in Proposition \ref{PROP:PB} and $\|A-\tilde{C}\tilde{U}_k^\dagger\tilde{R}\|$, (c) The ratio between our analytic error bounds  in Theorem \ref{THM:PB} and  $\|A-\tilde{C}\tilde{U}_k^\dagger\tilde{R}\|_2$.}
    \label{FIG:ErrorRatios}
\end{figure}

\end{experiment}
\begin{experiment}\label{Exp:ATA_Mrank}
Figure \ref{FIG:ErrorRatios} illustrates the error bounds derived in Section \ref{SEC:Perturbation} only for a fixed size and rank of $A$.  This begs the question: do these parameters affect the error bounds? We first test how the rank of $A$ affects the ratios between the analytic error bounds in Theorem \ref{THM:PB} and the true error $\|A-\tilde{C}\tilde{U}_k^\dagger\tilde{R}\|_2$.  To do this, we randomly generate a matrix $A$ of the size $500\times 500$, but force $A$ to be rank $k$, which varies between $1$ and $25$. As in Experiment \ref{EXP:ErrorBounds}, we perturb $A$ by a Gaussian random matrix with mean 0 and standard deviation $10^{-4}$. And  we choose 60 columns and rows via leverage scores, column/row lengths, and uniformly.  For each fixed rank $k$, we calculate the ratio of the error bound with the true error for 200 choices of columns and rows, and report the results in Figure \ref{fig:CUR_ARTA_Mrank}.  We see that the error bound degrades as the rank of $A$ increases.
\end{experiment}

\begin{experiment}
  Here, we test how the size of $A$ will influence the ratios between the estimated error bounds in Theorem \ref{THM:PB} and the true error $\|A-\tilde{C}\tilde{U}_k^\dagger\tilde{R}\|_2$. The setup of this experiment is the same as the one in Experiment \ref{Exp:ATA_Mrank} except that the rank of $A$ is $10$, but the size of $A$ varies from 100 to 500. The simulation results are shown in Fig \ref{fig:CUR_ARTA_Msize}, where we see that the size of $A$ does not influence the ratios overly much, but there is some indication that the bounds are better for larger matrices, which bodes well for utility in big data applications.

\begin{figure}[h]
\centering
\begin{subfigure}[b]{0.49\linewidth}
		\includegraphics[width=\textwidth]{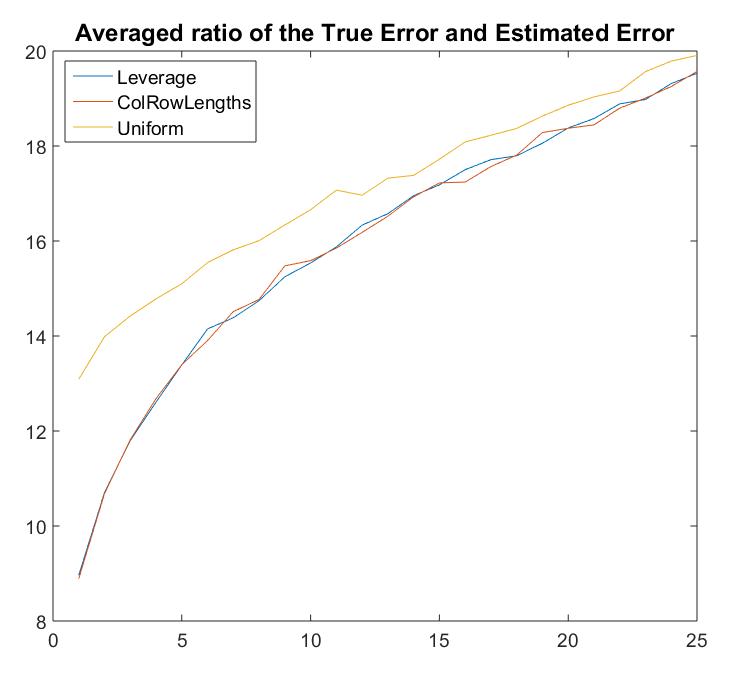}
		\caption{}	\label{fig:CUR_ARTA_Mrank}
	\end{subfigure}
	\begin{subfigure}[b]{0.49\linewidth}
		\includegraphics[width=\textwidth]{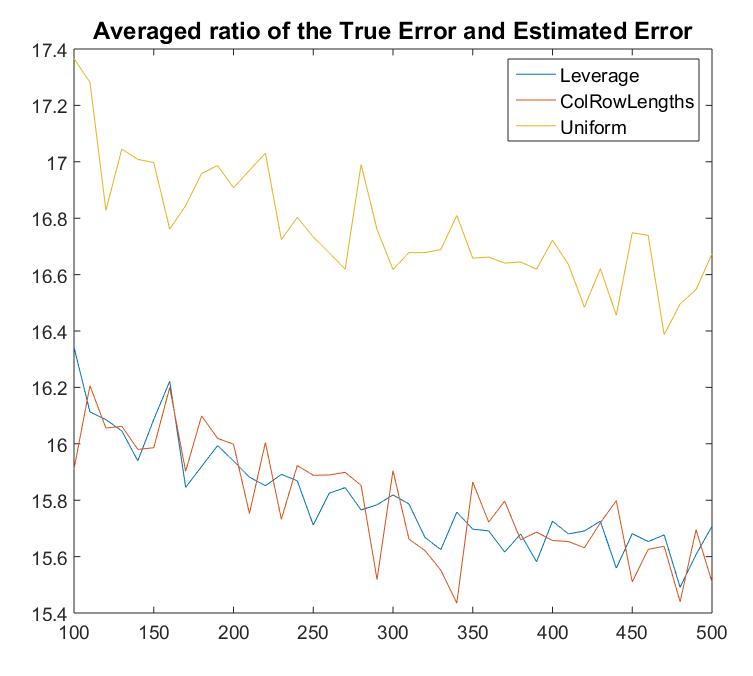}
		\caption{}	\label{fig:CUR_ARTA_Msize}
	\end{subfigure}
\caption{(a) The averaged ratio between our analytic error bounds in Theorem \ref{THM:PB} and  $\|A-\tilde{C}\tilde{U}_k^\dagger\tilde{R}\|_2$ by varying the rank but fixing the size of $A$. (b) Same ratio as (a) when varying the size of $A$ but fixing its rank.}
\end{figure}
\end{experiment}

\section{Rank estimation}\label{SEC:Rank}
The experiments in Section \ref{SEC:Numerics} indicate that having a good estimate for the rank of $A$ is crucial to obtaining a good CUR approximation to it.  There are many ways to do this, including the simple method of looking for a steep drop in the singular values of $\tilde{A}$.  Here, we give a simple algorithm derived from our previous analysis to estimate the rank of a matrix perturbed by Gaussian noise.

Again supposing that $\tilde{A}=A+E$ with $\rank(A)=k$ and $E$ being an i.i.d. Gaussian matrix with $N(0,\sigma^2)$ entries.  Suppose $U=A(I,J)$ such that $\rank(U)=k$, and hence $A=CU^\dagger R$.  Supposing we have selected rows and columns to form $\tilde{U}$, it follows from Theorem \ref{THM:Stewart} that if $s>k$, then 
\[\|\tilde{U}_s^\dagger\|_2\geq\frac{1}{\|\tilde{U}_s-U\|_2}\geq\frac{1}{2\|E(I,J)\|_2}.\]  The second inequality above follows from the same estimation as done in Section \ref{SEC:Perturbation}, which shows that $\|\tilde{U}_s-U\|_2\leq2\|E(I,J)\|_2.$

In \cite[p. 138]{taotopics} it is shown that an $n\times n$ Wigner matrix $M$ in which all off-diagonal entries have mean zero and unit variance, and the diagonal entries have mean zero and bounded variance, has the following property asymptotically almost surely:
\[
(1+o(1))\sqrt{n}\leq\|M\|_{2}\leq(1+o(1))n.
\]
By using the symmetrization technique, we find that for an $m\times n$ random Gaussian matrix $E$ with mean 0 and variance $\sigma^2$, 
\begin{equation}
\|E\|_{2}\leq (1+o(1))\sigma \sqrt{mn}.
\end{equation}
Therefore, if $\|\tilde{U}^\dagger_s\|_2\leq1/(2\sigma\sqrt{|I||J|})$, then $\rank(U)\geq s$.  We may use this fact to obtain an upper bound for the rank of $U$, which we present as Algorithm \ref{algorithm iterative}.

\begin{algorithm}[h!]
	\SetKwInOut{Output}{Output}
	
	{\text {\bf Input:} $\tilde{A}=A+E$ with $E_{i,j}\sim N(0,\sigma^2)$}
	
	\Output{Upper bound for the rank of matrix $A\in\mathbb{R}^{m\times n}$}
	
	{Form a $CUR$ approximation of $\tilde{A}$ as follows.}
      
	{Select $I\subset\{1,2,\ldots,m \}$ and $J\subset\{1,2,\ldots,n \}$ randomly w.p. $p_i=\frac{\| \tilde{A}(i,:)\|_2^2}{\|\tilde{A}\|_F^2}$ and $q_j=\frac{\|\tilde{A}(:,j)\|_2^2}{\|\tilde{A}\|_F^2}$, respectively}

	{Set $\tilde{C}=\tilde{A}(:,J)$, $\tilde{R}=\tilde{A}(I,:)$, and $\tilde{U}=\tilde{A}(I,J)$}
	
	{Compute the full SVD of $\tilde{U}$: $\tilde{U} = W\Sigma V^*$}
	
	
	
	{Set $K=\min\{|I|,|J|\}$.}
	
	\For{$i = \min\{ |I|,|J|\}$  \KwTo $1$ }
	{
	   {Compute	$error=\|\Sigma_i^\dagger\|_2$}
	   
	   \If{$error< 1/(2*\sqrt{|I||J|}\cdot\sigma).$}
	   {{Set $K=i$.}
	   
	   {Break.}}
	}
	\text{\bf Output} {$K$}
	\caption{Rank Estimation.}\label{algorithm iterative}
\end{algorithm}

\begin{remark}
The complexity of Algorithm \ref{algorithm iterative} is dominated by the complexity of finding the CUR approximation of $A$. We have an expected number of rows and columns chosen $(s,r)$, and then SVD($\tilde{U}$) will cost $O(\min\{sr^2,s^2r\})$, and the CUR approximation of $\tilde{A}$ has the cost of finding $\tilde{U}^\dagger$, which is the same as the SVD of $\tilde{U}$.
\end{remark}

Let us now briefly illustrate Algorithm \ref{algorithm iterative} in a simulation by considering the relation between the estimated rank and the variance of the noise.  We randomly generate a $500\times 500$ matrix $A$ of rank $20$, and perturb it by i.i.d. Gaussian noise with mean $0$ and standard deviation $\sigma$ (which will vary).  Then we uniformly randomly select 60 rows and columns to generate $\tilde{U}$.  The relationship between the estimated rank (via the output of Algorithm \ref{algorithm iterative}) and the standard deviation of the noise is shown in Figure \ref{FIG:RNV}. We see that when the standard deviation of  of the Gaussian noise is less than $10^{-3}$, the estimated rank is exactly the rank of the original noise-free matrix, and the bound quickly degrades subsequently.

\begin{figure}[h!]
\includegraphics[scale=0.5]{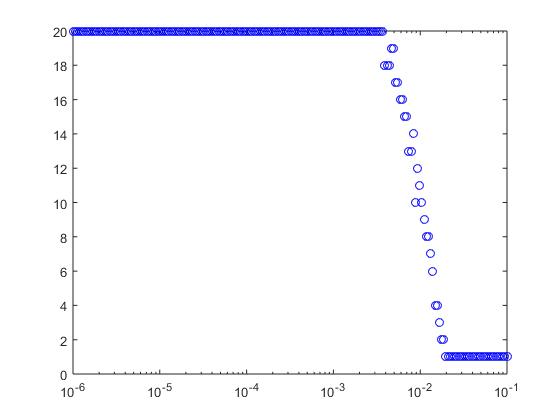}
\caption{The relation between the estimated ranks and the standard deviation of the noise.}\label{FIG:RNV}
\end{figure}

Figure \ref{FIG:RNN} shows the effect of the size of the submatrix $U$ on the rank estimation.  In each case, the noise is fixed, but the number of rows and columns increases.  Evidently, for low levels of noise, we see that choosing very close to 20 columns and rows yields a good rank estimation, while for larger noise, it is better to choose more rows and columns to form the CUR approximation.  For similar experimental results on CUR approximations, see \cite{SPECTS}

\begin{figure}
	\centering
	\begin{subfigure}[b]{0.49\textwidth}
		\includegraphics[width=\textwidth]{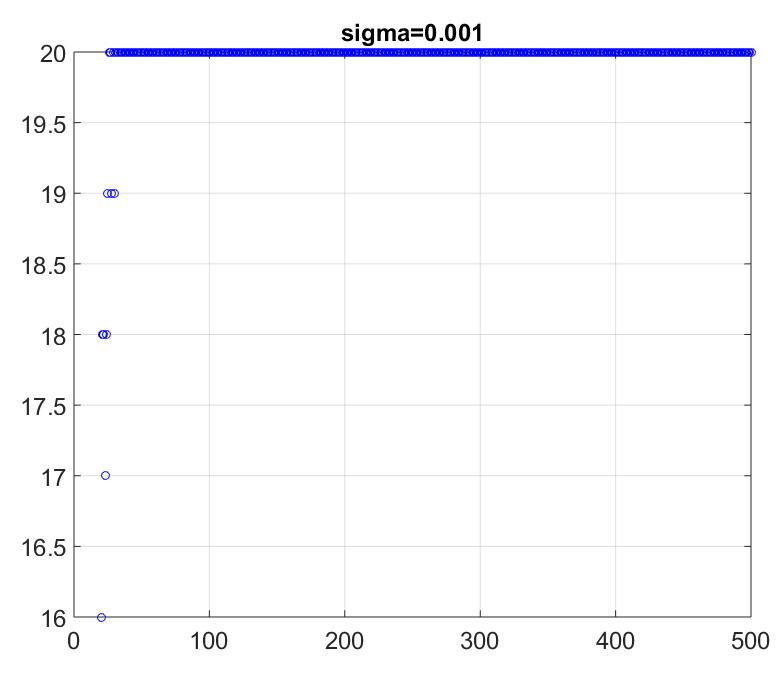}
		\caption{}
	\end{subfigure}
	\begin{subfigure}[b]{0.49\textwidth}
		\includegraphics[width=\textwidth]{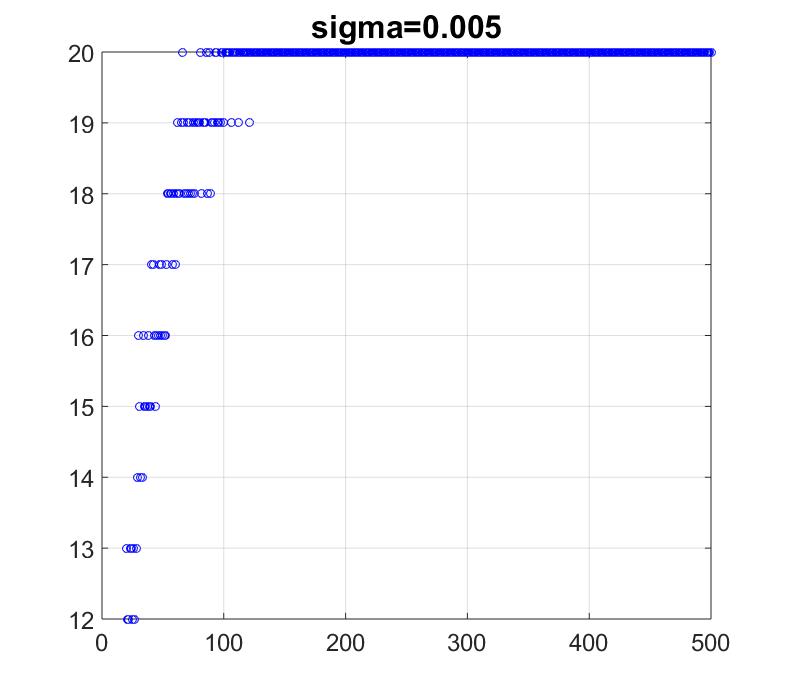}
		\caption{}
	\end{subfigure}
	\caption{The relation between the estimated ranks and the size of the submatrix $U$.}
	\label{FIG:RNN}
\end{figure}

\section*{Acknowledgements}  Initial work for this article was done while the K. H. was an Assistant Professor at Vanderbilt University.  K. H. is partially supported by the National Science Foundation TRIPODS program, grant number NSF CCF--1423411. LX.H. is partially supported by NSF Grant DMS-1322099.

We are indebted to Amy Hamm Design for making our sketches of Figures \ref{FIG:CUR}, and \ref{FIG:CURSpace2} a reality.  K. H. thanks Vahan Huroyan for many helpful discussions related to this work, and David Glickenstein and Jean-Luc Bouchot for comments on a preliminary version of the manuscript.

\bibliographystyle{plain}
\bibliography{HammHuang}

\begin{thebibliography}{10}

\bibitem{AHKS}
Akram Aldroubi, Keaton Hamm, Ahmet~Bugra Koku, and Ali Sekmen.
\newblock {CUR} decompositions, similarity matrices, and subspace clustering.
\newblock {\em Frontiers in Applied Mathematics and Statistics}, 4:65, 2019.

\bibitem{AltschulerGreedyCSSP}
Aditya Bhaskara, Afshin Rostamizadeh, Jason Altschuler, Morteza Zadimoghaddam,
  Thomas Fu, and Vahab Mirrokni.
\newblock Greedy column subset selection: New bounds and distributed
  algorithms.
\newblock ICML, 2016.

\bibitem{Bien}
Jacob Bien, Ya~Xu, and Michael~W Mahoney.
\newblock Cur from a sparse optimization viewpoint.
\newblock In {\em Advances in Neural Information Processing Systems}, pages
  217--225, 2010.

\bibitem{RCUR}
Andr{\'a}s Bodor, Istv{\'a}n Csabai, Michael~W Mahoney, and Norbert Solymosi.
\newblock r{CUR}: an {R} package for {CUR} matrix decomposition.
\newblock {\em BMC bioinformatics}, 13(1):103, 2012.

\bibitem{BoutsidisNearOptimalCSSP}
Christos Boutsidis, Petros Drineas, and Malik Magdon-Ismail.
\newblock Near-optimal column-based matrix reconstruction.
\newblock {\em SIAM Journal on Computing}, 43(2):687--717, 2014.

\bibitem{Boutsidis2009}
Christos Boutsidis, Michael~W Mahoney, and Petros Drineas.
\newblock An improved approximation algorithm for the column subset selection
  problem.
\newblock In {\em Proceedings of the twentieth annual ACM-SIAM symposium on
  Discrete algorithms}, pages 968--977. SIAM, 2009.

\bibitem{BoutsidisOptimalCUR}
Christos Boutsidis and David~P Woodruff.
\newblock Optimal {CUR} matrix decompositions.
\newblock {\em SIAM Journal on Computing}, 46(2):543--589, 2017.

\bibitem{CaiafaTensorCUR}
Cesar~F. Caiafa and Andrzej Cichocki.
\newblock Generalizing the column-row matrix decomposition to multi-way arrays.
\newblock {\em Linear Algebra and its Applications}, 433(3):557 -- 573, 2010.

\bibitem{CandesRomberg}
Emmanuel Cand{\`e}s and Justin Romberg.
\newblock Sparsity and incoherence in compressive sampling.
\newblock {\em Inverse problems}, 23(3):969, 2007.

\bibitem{DemanetWu}
Jiawei Chiu and Laurent Demanet.
\newblock Sublinear randomized algorithms for skeleton decompositions.
\newblock {\em {SIAM} Journal on Matrix Analysis and Applications},
  34(3):1361--1383, 2013.

\bibitem{Civril}
Ali {\c{C}}ivril.
\newblock Column subset selection problem is {UG}-hard.
\newblock {\em Journal of Computer and System Sciences}, 80(4):849--859, 2014.

\bibitem{CosteiraKanade}
Jo{\~a}o~Paulo Costeira and Takeo Kanade.
\newblock A multibody factorization method for independently moving objects.
\newblock {\em International Journal of Computer Vision}, 29(3):159--179, 1998.

\bibitem{Deshpande2010}
Amit Deshpande and Luis Rademacher.
\newblock Efficient volume sampling for row/column subset selection.
\newblock In {\em Foundations of Computer Science (FOCS), 2010 51st Annual IEEE
  Symposium on}, pages 329--338. IEEE, 2010.

\bibitem{DKMIII}
Petros Drineas, Ravi Kannan, and Michael~W Mahoney.
\newblock Fast monte carlo algorithms for matrices {III}: Computing a
  compressed approximate matrix decomposition.
\newblock {\em SIAM Journal on Computing}, 36(1):184--206, 2006.

\bibitem{DM05}
Petros Drineas and Michael~W Mahoney.
\newblock On the {N}ystr{\"o}m method for approximating a {G}ram matrix for
  improved kernel-based learning.
\newblock {\em journal of machine learning research}, 6(Dec):2153--2175, 2005.

\bibitem{DMM08}
Petros Drineas, Michael~W Mahoney, and S~Muthukrishnan.
\newblock Relative-error {CUR} matrix decompositions.
\newblock {\em SIAM Journal on Matrix Analysis and Applications},
  30(2):844--881, 2008.

\bibitem{SSC}
Ehsan Elhamifar and Rene Vidal.
\newblock Sparse subspace clustering: Algorithm, theory, and applications.
\newblock {\em IEEE transactions on pattern analysis and machine intelligence},
  35(11):2765--2781, 2013.

\bibitem{Gantmacher}
Feliks~R Gantmacher.
\newblock Matrix theory.
\newblock {\em Chelsea, New York}, 21, 1959.

\bibitem{GittensMahoney}
Alex Gittens and Michael~W Mahoney.
\newblock Revisiting the {N}ystr{\"o}m method for improved large-scale machine
  learning.
\newblock {\em The Journal of Machine Learning Research}, 17(1):3977--4041,
  2016.

\bibitem{GolubVanLoan}
Gene~H. Golub and Charles~F. van Loan.
\newblock {\em Matrix Computations}.
\newblock The Johns Hopkins University Press, Baltimore, fourth edition, 2013.

\bibitem{Goreinov}
S.~A. Gore\u\i{}nov, N.~L. Zamarashkin, and E.~E. Tyrtyshnikov.
\newblock Pseudo-skeleton approximations of matrices.
\newblock {\em Dokl. Akad. Nauk}, 343(2):151--152, 1995.

\bibitem{Goreinov2}
Sergei~A. Gore\u\i{}nov, Eugene~E. Tyrtyshnikov, and Nickolai~L. Zamarashkin.
\newblock A theory of pseudoskeleton approximations.
\newblock {\em Linear algebra and its applications}, 261(1-3):1--21, 1997.

\bibitem{Goreinov3}
Sergei~A Gore\u\i{}nov, Nikolai~Leonidovich Zamarashkin, and
  Evgenii~Evgen’evich Tyrtyshnikov.
\newblock Pseudo-skeleton approximations by matrices of maximal volume.
\newblock {\em Mathematical Notes}, 62(4):515--519, 1997.

\bibitem{tropp}
Nathan Halko, Per-Gunnar Martinsson, and Joel~A Tropp.
\newblock Finding structure with randomness: Probabilistic algorithms for
  constructing approximate matrix decompositions.
\newblock {\em SIAM review}, 53(2):217--288, 2011.

\bibitem{KannanVempala}
Ravindran Kannan and Santosh Vempala.
\newblock Randomized algorithms in numerical linear algebra.
\newblock {\em Acta Numerica}, 26:95--135, 2017.

\bibitem{Khot}
Subhash Khot.
\newblock On the power of unique 2-prover 1-round games.
\newblock In {\em Proceedings of the Thiry-fourth Annual ACM Symposium on
  Theory of Computing}, STOC '02, pages 767--775, New York, NY, USA, 2002. ACM.

\bibitem{KhouryScheidegger}
Marc Khoury, Yifan Hu, Shankar Krishnan, and Carlos Scheidegger.
\newblock Drawing large graphs by low-rank stress majorization.
\newblock In {\em Computer Graphics Forum}, volume~31, pages 975--984. Wiley
  Online Library, 2012.

\bibitem{CURTPAMI}
C.~Li, X.~Wang, W.~Dong, J.~Yan, Q.~Liu, and H.~Zha.
\newblock Joint active learning with feature selection via cur matrix
  decomposition.
\newblock {\em IEEE Transactions on Pattern Analysis and Machine Intelligence},
  pages 1--1, 2018.

\bibitem{LiDeterministicCSSP}
Xuelong Li and Yawei Pang.
\newblock Deterministic column-based matrix decomposition.
\newblock {\em IEEE Transactions on Knowledge and Data Engineering},
  22(1):145--149, 2010.

\bibitem{DMPNAS}
Michael~W Mahoney and Petros Drineas.
\newblock {CUR} matrix decompositions for improved data analysis.
\newblock {\em Proceedings of the National Academy of Sciences},
  106(3):697--702, 2009.

\bibitem{MOA_2011}
Albert~W. Marshall, Ingram Olkin, and Barry~C. Arnold.
\newblock {\em Inequalities: Theory of majorization and its applications}.
\newblock Springer Series in Statistics. Springer-Verlag New York, 2 edition,
  2011.

\bibitem{Mirsky}
Leon Mirsky.
\newblock Symmetric gauge functions and unitarily invariant norms.
\newblock {\em The quarterly journal of mathematics}, 11(1):50--59, 1960.

\bibitem{OrdozgoitiCSSP}
Bruno Ordozgoiti, Sandra~G{\'o}mez Canaval, and Alberto Mozo.
\newblock Iterative column subset selection.
\newblock {\em Knowledge and Information Systems}, 54(1):65--94, 2018.

\bibitem{Osinsky2018}
AI~Osinsky and NL~Zamarashkin.
\newblock Pseudo-skeleton approximations with better accuracy estimates.
\newblock {\em Linear Algebra and its Applications}, 537:221--249, 2018.

\bibitem{Penrose56}
R.~Penrose.
\newblock On best approximate solutions of linear matrix equations.
\newblock {\em Mathematical Proceedings of the Cambridge Philosophical
  Society}, 52(1):17–19, 1956.

\bibitem{BeckerNystrom}
Farhad Pourkamali-Anaraki and Stephen Becker.
\newblock Improved fixed-rank {N}ystr\"{o}m approximation via qr decomposition:
  Practical and theoretical aspects.
\newblock {\em arXiv preprint arXiv:1708.03218}, 2017.

\bibitem{RudelsonPrivate}
Mark Rudelson.
\newblock {Personal Communication}, 2019.

\bibitem{Rudelson_2007}
Mark Rudelson and Roman Vershynin.
\newblock Sampling from large matrices: an approach through geometric
  functional analysis.
\newblock {\em Journal of the {ACM}}, 54(4):21--es, jul 2007.

\bibitem{SPECTS}
Ali Sekmen, Akram Aldroubi, Ahmet~Bugra Koku, and Keaton Hamm.
\newblock Matrix reconstruction: Skeleton decomposition versus singular value
  decomposition.
\newblock In {\em 2017 International Symposium on Performance Evaluation of
  Computer and Telecommunication Systems (SPECTS)}, pages 1--8. IEEE, 2017.

\bibitem{Shitov}
Yaroslav Shitov.
\newblock Column subset selection is {NP}-complete.
\newblock {\em arXiv preprint arXiv:1701.02764}, 2017.

\bibitem{SorensenDEIMCUR}
Danny~C Sorensen and Mark Embree.
\newblock A {DEIM} induced {CUR} factorization.
\newblock {\em SIAM Journal on Scientific Computing}, 38(3):A1454--A1482, 2016.

\bibitem{Stewart_1977}
G.~W. Stewart.
\newblock On the perturbation of pseudo-inverses, projections and linear least
  squares problems.
\newblock {\em {SIAM} Review}, 19(4):634--662, oct 1977.

\bibitem{stewart_minimizer}
GW~Stewart.
\newblock Four algorithms for the the efficient computation of truncated
  pivoted qr approximations to a sparse matrix.
\newblock {\em Numerische Mathematik}, 83(2):313--323, 1999.

\bibitem{Strang}
Gilbert Strang, Gilbert Strang, Gilbert Strang, and Gilbert Strang.
\newblock {\em Introduction to linear algebra}, volume~3.
\newblock Wellesley-Cambridge Press Wellesley, MA, 1993.

\bibitem{WZ_2013}
S.Wang and Z.Zhang.
\newblock Improving {CUR} matrix decomposition and the {N}ystr\"{o}m
  approximation via adaptive sampling.
\newblock {\em The Journal of Machine Learning Research}, 14:2729--2769,
  January 2013.

\bibitem{taotopics}
Terence Tao.
\newblock {\em Topics in random matrix theory}, volume 132.
\newblock American Mathematical Soc., 2012.

\bibitem{ThompsonInterlacing}
Robert~C Thompson.
\newblock Principal submatrices ix: Interlacing inequalities for singular
  values of submatrices.
\newblock {\em Linear Algebra and its Applications}, 5(1):1--12, 1972.

\bibitem{Hopkins}
Roberto Tron and Ren{\'e} Vidal.
\newblock A benchmark for the comparison of 3-d motion segmentation algorithms.
\newblock In {\em Computer Vision and Pattern Recognition, 2007. CVPR'07. IEEE
  Conference on}, pages 1--8. IEEE, 2007.

\bibitem{tropp2009column}
Joel~A Tropp.
\newblock Column subset selection, matrix factorization, and eigenvalue
  optimization.
\newblock In {\em Proceedings of the Twentieth Annual ACM-SIAM Symposium on
  Discrete Algorithms}, pages 978--986. Society for Industrial and Applied
  Mathematics, 2009.

\bibitem{TroppNystrom}
Joel~A Tropp, Alp Yurtsever, Madeleine Udell, and Volkan Cevher.
\newblock Fixed-rank approximation of a positive-semidefinite matrix from
  streaming data.
\newblock In {\em Advances in Neural Information Processing Systems}, pages
  1225--1234, 2017.

\bibitem{LowRank}
Madeleine Udell and Alex Townsend.
\newblock Why are big data matrices approximately low rank?
\newblock {\em SIAM Journal on Mathematics of Data Science}, 1(1):144--160,
  2019.

\bibitem{VoroninMartinsson}
Sergey Voronin and Per-Gunnar Martinsson.
\newblock Efficient algorithms for {CUR} and interpolative matrix
  decompositions.
\newblock {\em Advances in Computational Mathematics}, 43(3):495--516, 2017.

\bibitem{Xu15}
Miao Xu, Rong Jin, and Zhi-Hua Zhou.
\newblock {CUR} algorithm for partially observed matrices.
\newblock In {\em International Conference on Machine Learning}, pages
  1412--1421, 2015.

\bibitem{Yang15}
Jiyan Yang, Oliver R\"{u}bel, Michael~W Mahoney, and Benjamin~P Bowen.
\newblock Identifying important ions and positions in mass spectrometry imaging
  data using {CUR} matrix decompositions.
\newblock {\em Analytical chemistry}, 87(9):4658--4666, 2015.

\bibitem{Yang}
Tianbao Yang, Lijun Zhang, Rong Jin, and Shenghuo Zhu.
\newblock An explicit sampling dependent spectral error bound for column subset
  selection.
\newblock In {\em International Conference on Machine Learning}, pages
  135--143, 2015.

\bibitem{Yip14}
Ching-Wa Yip, Michael~W Mahoney, Alexander~S Szalay, Istv{\'a}n Csabai,
  Tam{\'a}s Budav{\'a}ri, Rosemary~FG Wyse, and Laszlo Dobos.
\newblock Objective identification of informative wavelength regions in galaxy
  spectra.
\newblock {\em The Astronomical Journal}, 147(5):110, 2014.

\end{thebibliography}

\appendix

\section{Proof of Theorem \ref{THM:Stability}}\label{SEC:AppendixProofStability}

The first ingredient in the proof is the following:
\begin{theorem}[{\cite[Theorem 3.1]{Rudelson_2007}}]\label{THM:RVLargeNumbers}
Let $y$ be a random vector in $\K^n$ which is uniformly bounded almost everywhere, i.e. $\|y\|_2\leq M$.  Assume for normalization that $\|\E(y\otimes y)\|_2\leq1$.  Let $y_1,\dots,y_d$ be independent copies of $y$.  Let \[ a:= C_0\sqrt{\frac{\log d}{d}}M. \] Then
\begin{enumerate}
    \item[(i)] If $a<1$, then \[ \E\left\|\frac1d\sum_{i=1}^d y_i\otimes y_i-\E(y\otimes y) \right\|_2\leq a;\]
    \item[(ii)] For every $t\in(0,1)$, \[ \Prob\left\{\left\|\frac1d\sum_{i=1}^d y_i\otimes y_i-\E(y\otimes y)\right\|_2>t\right\} \leq 2\exp(-ct^2/a^2). \]
\end{enumerate}
\end{theorem}

Note that Theorem \ref{THM:RVLargeNumbers} was proved in \cite{Rudelson_2007} for $y\in\R^n$; however, their proof is valid without change for complex vectors, which we need for our application \cite{RudelsonPrivate}.

\begin{proof}[Proof of Theorem \ref{THM:Stability}]
Without loss of generality, suppose that $\|A\|_2=1$.  Let $x_i$ be the rows of $A$ so that $A^*A = \sum_{i=1}^m x_i\otimes x_i$.  Define the random vector $y$ via
\[ \Prob\left(y = \frac{1}{\sqrt{\tilde{p}_i}}x_i\right) = \tilde{p}_i.\]
Note that by assumption on $\tilde{p}$, $\Prob(y=x_i) = 0$ only if $x_i=0$.  Let $y_1,\dots,y_d$ be independent copies of $y$, and let $\hat{A}$ be the matrix whose rows are $\frac{1}{\sqrt{d}}y_i$. Then we have $\hat{A}^*\hat{A} = \frac1d\sum_{i=1}^d y_i\otimes y_i$, and  $\E(y\otimes y) = A^*A$; indeed
\[\E(y\otimes y) = \sum_{i=1}^m \frac{1}{\sqrt{\tilde{p}_i}}x_i\otimes \frac{1}{\sqrt{\tilde{p}_i}}x_i\tilde{p}_i = \sum_{i=1}^m x_i\otimes x_i = A^*A.  \]
Now by assumption on $\tilde{p}$, we may choose \[\|y\|_2 = \frac{\|x_i\|_2}{\alpha_i\|x_i\|_2}\|A\|_F \leq \frac1\alpha\|A\|_F = \frac{r}{\alpha}:=:M.\]  Applying Theorem \ref{THM:RVLargeNumbers} with the assumption on $d$ yields (as in \cite{Rudelson_2007})
\[ a = \frac{1}{\alpha}C\left(\frac{\log d}{d}r\right)^\frac12 \leq \frac{\eps^2\sqrt{\delta}}{2\alpha}.\]
This quantity is thus bounded by 1 provided $\frac{\eps^2\sqrt{\delta}}{2}\leq\alpha$. In this event, Theorem \ref{THM:RVLargeNumbers} $(ii)$ implies that if $t=\frac{\eps^2}{2}$, then \[\|A^*A-\hat{A}^*\hat{A}\|_2\leq\frac{\eps^2}{2} \]
with probability at least $1-2\exp(-\frac{c\alpha^2}{\delta})$.  Note also that by the assumption on $\eps,\delta$, and $\alpha$, we have $\frac{\alpha^2}{\delta}>\frac{\eps^4}{4}$, whence the given event holds with probability at least $1-2\exp(-c\eps^4)$.
\end{proof}

\end{document}